\documentclass[a4paper,12pt,draft]{amsart}
\usepackage[margin=30mm]{geometry}
\usepackage{amssymb, amsmath, amsthm, amscd}

\usepackage[T1]{fontenc}
\usepackage{eucal,mathrsfs,dsfont}%gives nicer set-names. Use \mathds instead of \mathds
\usepackage{color}%cyan,red;magenta,green;yellow,blue
\definecolor{mno}{rgb}{0.5,0.1,0.5}

\newcommand{\R}{\mathds R}

\newcommand{\Pp}{\mathds P}
\newcommand{\Ee}{\mathds E}

\newcommand{\I}{\mathds 1}

\def\<{\langle}
\def\>{\rangle}
\newcommand{\var}{\operatorname{Var}}

\newtheorem{theorem}{Theorem}[section]
\newtheorem{lemma}[theorem]{Lemma}
\newtheorem{proposition}[theorem]{Proposition}
\newtheorem{corollary}[theorem]{Corollary}

\theoremstyle{definition}

\newtheorem{remark}[theorem]{Remark}

\makeatletter % '@' is now a normal "letter" for TeX

\@addtoreset{equation}{section}
\makeatother % '@' is restored as a "non-letter" character for TeX

\begin{document}
\allowdisplaybreaks
\title[Exponential ergodicity for McKean-Vlasov SDEs with L\'{e}vy noise]
{\bfseries Exponential ergodicity for SDEs and McKean-Vlasov processes with L\'{e}vy noise}

\author{Mingjie Liang\qquad Mateusz B. Majka \qquad Jian Wang}

\thanks{\emph{M.\ Liang:}
College of Information Engineering,  Sanming
University, 365004, Sanming, P.R. China. \texttt{liangmingjie@aliyun.com}}
\thanks{\emph{M.\ B.\ Majka:}
School of Mathematical and Computer Sciences, Heriot-Watt University, Edinburgh, EH14 4AS, UK. \texttt{m.majka@hw.ac.uk}}

\thanks{\emph{J.\ Wang:}
School of Mathematics and Computer Science \& Fujian Key Laboratory of Mathematical Analysis and Applications (FJKLMAA) \& Center for Applied Mathematics of Fujian Province (FJNU), Fujian Normal
University, 350007, Fuzhou, P.R. China.
\texttt{jianwang@fjnu.edu.cn}}
\date{}

\begin{abstract}
We study exponential ergodicity of a broad class of stochastic processes whose dynamics are governed by pure jump L\'{e}vy noise. In the first part of the paper we focus on solutions of stochastic differential equations (SDEs) whose drifts satisfy general Lyapunov-type conditions. By applying techniques that combine couplings, appropriately constructed $L^1$-Wasserstein distances and Lyapunov functions, we show exponential convergence of solutions of such SDEs to their stationary distributions, both in the total variation and Wasserstein distances. The second part of the paper is devoted to SDEs of McKean-Vlasov type with distribution dependent drifts. We prove a uniform in time propagation of chaos result, providing quantitative bounds on convergence rate of interacting particle systems with L\'{e}vy noise to the corresponding McKean-Vlasov SDE. Then, extending our techniques from the first part of the paper, we obtain results on
 exponential
 ergodicity of solutions of McKean-Vlasov SDEs, under general conditions on the drift and the driving noise.

\medskip


\medskip

\end{abstract}

\maketitle
\allowdisplaybreaks

\section{Introduction and main results}\label{section1}

In recent years, the study of ergodicity
for solutions of stochastic differential equations (SDEs) with L\'{e}vy noise has attracted considerable interest (see e.g.\ \cite{Mas, Kul, JWang, Maj, Lwcw, Maj2, APS1, APS2}).
The most widespread tools in this area of research are the Meyn-Tweedie approach based on showing irreducibility and aperiodicity of the
 associated process (\cite{Mas, APS1, APS2}), and the coupling approach (\cite{Maj, Lwcw, Maj2}). For SDEs in $\R^d$ of the form
\begin{equation}\label{SDE1}
dX_t = b(X_t)dt + dZ_t \,,
\end{equation}
where $b: \R^d \to \R^d$ is the drift and $(Z_t)_{t\ge0}$ is a $d$-dimensional pure jump
L\'{e}vy process, the Meyn-Tweedie approach allows for studying
exponential
convergence under rather general conditions on the drift, both in the total variation (see e.g.\ \cite[Corollary 5.2]{APS1}) and the Wasserstein distances (see \cite[Theorem 1.3]{APS2} with the special case of $a(x)\equiv0$ and $\nu(x,dy)=\nu(dy)$ for all $x\in \R^d$). Namely, the authors of \cite{APS1, APS2} considered conditions of the form
\begin{equation}\label{drift1} \langle x, b(x) \rangle \leq - \lambda |x|^2
\end{equation} for some constant $\lambda > 0$ and all $x \in \R^d$ with large enough $|x|$. The downside is that in order to show irreducibility and aperiodicity of solutions to (\ref{SDE1}), both \cite{APS1} and \cite{APS2} cover only pure jump L\'{e}vy noises that are either compound Poisson processes or contain a symmetric $\alpha$-stable component;
see \cite[Section 4]{APS1} and \cite[Section 3.3]{APS2} for details. Moreover, usually it is difficult to quantify
 convergence rates obtained via such methods (cf.\ \cite[Section 3.1]{EGZ}). On the other hand, by applying the coupling technique, \cite{Lwcw, Maj, Maj2} considered SDEs with a much more general class of driving noises, while requiring that the drift $b(x)$ in (\ref{SDE1}) satisfies conditions of the form
\begin{equation}\label{drift2}
\langle x - y , b(x) - b(y) \rangle \leq - K|x-y|^2
\end{equation}
for all $x$, $y \in \R^d$ such that $|x-y| > R$, with some constants $K$, $R > 0$. Note that (\ref{drift2}) obviously implies (\ref{drift1}). In the present paper, we combine the coupling approach from \cite{Lwcw} with techniques recently developed for diffusions in \cite{EGZ}, in order to study
 exponential convergence of solutions to SDEs of the form (\ref{SDE1}) under conditions on the drift such as (\ref{drift1}), and for a very general class of L\'{e}vy processes $(Z_t)_{t \geq 0}$, thus considerably extending the results on exponential ergodicity from all the papers \cite{APS1,APS2,Lwcw,Maj,Maj2}. To this end, we will use the $L^1$-Wasserstein distance defined for any probability measures $\mu_1$, $\mu_2$ on $\R^d$ by
$$
W_1(\mu_1,\mu_2)=\inf_{\Pi\in \mathscr{C}(\mu_1,\mu_2)}\int_{\R^d\times\R^d} |x-y| \,d\Pi(x,y),
$$
where $\mathscr{C}(\mu_1,\mu_2)$ is the family of all measures on
$\R^d\times\R^d$ having $\mu_1$ and $\mu_2$ as marginals.
We will also denote by $\|\mu_1-\mu_2\|_{\var}$
the total variation distance between probability measures $\mu_1$ and $\mu_2$.
Everywhere in this paper we will assume that \emph{\eqref{SDE1} has a unique strong solution}.
This is the case for any L\'{e}vy noise $(Z_t)_{t \geq 0}$ if we assume that $b(x)$ is continuous and that
\begin{equation}\label{drift3}
\langle x - y , b(x) - b(y) \rangle \leq C|x-y|^2
\end{equation}
for all $x$, $y \in \mathbb{R}^d$, with some constant $C > 0$
(see e.g.\ \cite{GK, Maj3}; see also the discussion in Section 2 for results under different assumptions on the drift).
Moreover, for the L\'{e}vy measure $\nu$ of the process $(Z_t)_{t \geq 0}$ we define
\begin{equation}\label{defJ}
J(r) := \inf_{|x| \leq r} \left( \nu \wedge (\delta_x \ast \nu) \right) (\R^d),\quad r>0.
\end{equation}
We have the following result.

\begin{theorem}\label{mainSDEresult}
	Let $(X_t)_{t \geq 0}$ be
 the unique strong
 solution to \eqref{SDE1} with the drift $b(x)$ satisfying  \eqref{drift1} and \eqref{drift3}. Suppose that for the L\'{e}vy measure $\nu$ of $(Z_t)_{t \geq 0}$,
\begin{equation}\label{e:conitionv}\int_{\{|z|\ge1\}}|z|\,\nu(dz)<\infty\end{equation} and
 there is a constant $\kappa_0 \in (0,1]$ such that $J(\kappa_0) > 0$. Let $P_t(x,\cdot)$ be the law of $X_t$ for $t > 0$ when $X_0 = x$ for some $x \in \R^d$. Then there exist a unique invariant probability measure $\mu$ for $(X_t)_{t \geq 0}$ and a constant $\lambda > 0$ such that
	\begin{equation*}
	\| P_t(x,\cdot) - \mu \|_{\var} \leq c_1(x) e^{-\lambda t}
	\end{equation*}
	and
	\begin{equation*}
	W_1(P_t(x,\cdot), \mu) \leq c_2(x) e^{-\lambda t}
	\end{equation*}
	hold for all $x\in \R^d$ and $t > 0$, where $c_1(x)$ and $c_2(x)$ are non-negative measurable functions on $\R^d$.
\end{theorem}

Note that the assumption that $J(\kappa_0) > 0$ for some $\kappa_0 \in (0,1]$ is very weak (see e.g.\ \cite[Proposition 1.5]{SW2} and \cite[Remark 1.7]{Maj}), and hence Theorem \ref{mainSDEresult} covers a much larger class of L\'{e}vy processes than \cite{APS1,APS2,Lwcw,Maj,Maj2}.
On the other hand, we note that exponential ergodicity in terms of the total variation norm for solutions to (more general) SDEs with jump noise has been studied in \cite{Kul} via the local Doeblin condition, which is expressed in terms of the transition probability of the associated process (see condition {\bf (LD)} in \cite{Kul}). It was shown in \cite[Theorem 1.3]{Kul} that the local Doeblin condition can be derived from non-degeneracy condition {\bf N} and topological irreducibility condition {\bf S}, both of which essentially require non-degenerate properties of the L\'evy measure $\nu$ near zero; see \cite[Section 4]{Kul} for details.
In view of this, it seems that our assumption that $J(\kappa_0)>0$ for some $\kappa_0\in (0,1]$ in Theorem \ref{mainSDEresult} (which can be interpreted as requiring a sufficient mass of the L\'evy measure near zero)
looks simpler and more direct.
 We also remark that conditions (\ref{drift1}) and (\ref{drift3}) can be weakened even further (cf.\ Theorem \ref{th1:le} below). All the technical details and the proof of Theorem \ref{mainSDEresult} can be found in Section \ref{suce1}.
Besides, we would like to mention that \eqref{e:conitionv} implies that for any $x\in \R^d$ and $t>0$, $\Ee^x|X_t|<\infty$, which is needed for proving the exponential ergodicity of the process $(X_t)_{t\ge0}$ in terms of the $L^1$-Wasserstein distance. However, \eqref{e:conitionv} can be relaxed (for example, to the condition $\int_{\{|z|\ge1\}}|z|^\alpha\,\nu(dz)<\infty$ with some $\alpha\in (0,1)$; see \cite{W08, W12}) if one is interested in studying exponential ergodicity only in terms of the total variation norm.

\smallskip

The second part of this paper is devoted to stochastic differential equations on $\R^d$ of McKean-Vlasov type, i.e.,
\begin{equation}\label{non-1}
\begin{cases}
dX_t=b(X_t, \mu_t)\,dt+dZ_t, &X_0 \sim \mu_0,\\
\mu_t=\operatorname{Law}(X_t),
\end{cases}
\end{equation}
where  $b: \R^d \times \mathscr{P}(\R^d) \to \R^d$ is a measurable
function and $(Z_t)_{t\ge0}$ is a $d$-dimensional pure jump L\'evy
process. Here $\mathscr{P}(\R^d)$ denotes the family of all
probability measures on $\R^d$. The solution $(X_t)_{t\ge0}$ to
\eqref{non-1} is a nonlinear Markov process \cite{Kol} in the sense of
McKean-Vlasov, i.e.,\ the associated transition function depends
both on the current state $X_t$ and on the law of $X_t$. Equations
of the form \eqref{non-1} driven by Brownian motion have recently
attracted considerable attention, see e.g. \cite{BR, DST, EGZ, HSS, HRW,
Wang} and the references therein, as well as the monograph
\cite{Carmona}. The case in which the driving process has jumps has been much less studied, however, with the foundations set by \cite{JMW},
several papers on the topic have appeared in recent years, see e.g.
\cite{AP, ADF, BCD, CHM, HaoLi, Li, MSSZ, Song}. By analogy to the
relation between McKean-Vlasov SDEs with Brownian noise and
nonlinear Fokker-Planck equations \cite{BR, HRW}, McKean-Vlasov SDEs
with jumps can be related to nonlocal integral-PDEs \cite{HaoLi,
JMW, Li}. They have also found applications in areas such as
financial mathematics \cite{BCD} and
neural  networks \cite{MSSZ}.
Such equations, regardless of the driving noise, are known in the
literature under numerous different names, including McKean-Vlasov
SDEs \cite{ADF, Bu, EGZ, HSS, Gra}, mean-field SDEs \cite{CHM,
HaoLi,Li, Song}, distribution dependent SDEs (DDSDEs) \cite{BR,
Wang} and nonlinear SDEs \cite{JMW}. The choice of the name depends
on the preferred interpretation of the process $(X_t)_{t\ge0}$ and
different names are often used interchangeably.

  A typical example of the drift in
\eqref{non-1} is
\begin{equation}\label{e:dri}b(x, \mu)=b_1(x) +\displaystyle \int b_2(x, z)\,\mu(dz)\end{equation} for some measurable functions $b_1: \R^d \to \R^d$ and $b_2: \R^d \times \R^d \to \R^d$.
In such case the corresponding McKean-Vlasov SDE arises naturally as
a marginal limit as $n\to \infty $ of the mean-field interacting
particle system
\begin{equation}\label{e:mean}dX_t^{i,n}=b_1(X_t^{i,n})\,dt+\frac{1}{n}\sum_{j=1}^nb_2(X_t^{i,n}, X_t^{j,n})\,dt +dZ^i_t,\quad i=1,\ldots,n\end{equation} driven by independent L\'evy processes $(Z^i_t)_{t\ge0}$. This property was named \emph{propagation of chaos}
 by Kac \cite{Kac} and was further developed by Sznitman \cite{Sznit}.
  Propagation of chaos has been studied extensively in the Brownian setting, see e.g.
  \cite{CGM, DEGZ, Me} and the references therein. In the L\'{e}vy jump case,
  the result of such kind
  has been proved in \cite{JMW} for McKean-Vlasov (nonlinear) SDEs
  even with distribution dependent noise coefficients.
  See also \cite{ADF} for a recent study on nonlinear jump diffusions with finite jump activity. However, the bounds obtained in \cite{ADF, JMW} are uniform only on a finite time interval. In the present paper, we adapt the method from \cite{DEGZ} to show uniform in time convergence,
 as well as its explicit quantitative rate,
  of the particle system \eqref{e:mean} in the $L^1$-Wasserstein distance.

We work under relatively weak assumptions on the drift and the driving L\'{e}vy process. Namely, \emph{for the term $b_1(x)$ in the drift $b(x,\mu)$ given by \eqref{e:dri} we assume
\begin{itemize}
    \item[(1-i)]
$b_1(x)$ is continuous on $\R^d$, and
there are constants $K_{1,b_1},r_{b_1}\ge 0$ and $K_{2,b_1}>0$ such that for any $x,y\in \R^d$, \begin{equation}\label{cond:drif1}\langle b_1(x)-b_1(y),x-y\rangle \le
    K_{1,b_1}|x-y|^2\I_{\{|x-y|\le r_{b_1}\}}-K_{2,b_1}|x-y|^2\I_{\{|x-y| > r_{b_1}\}}.\end{equation}
\end{itemize}
Furthermore, for
the term $b_2(x,\mu)$ in \eqref{e:dri} we require
\begin{itemize}
    \item[(1-ii)] There exists a function $\tilde b_2 : \R^d \to \R^d$ with $\tilde b_2(0) = 0$ such that $b_2(x,z) = \tilde b_2(x-z)$ for any $x$, $z \in \R^d$, and there is a constant $K_{\tilde b_2}>0$ such that
    for all $x$, $y \in \R^d$,
    \begin{equation}\label{coe-b3}|\tilde b_2(x)-\tilde b_2(y)|\le K_{\tilde b_2}|x-y|.\end{equation}
\end{itemize}
Finally, for the L\'{e}vy measure $\nu$ of the L\'{e}vy process $(Z_t)_{t \geq 0}$ we consider the function $J(r)$ defined in  \eqref{defJ} and we assume the following conditions:
\begin{itemize}
    \item[(1-iii)] There exists a constant $\alpha \in (0,1)$ such that
    \begin{equation}\label{assumption:concentrationLevymeasure}
    \lim_{r\to 0}\inf_{s\in (0,r]} J(s)s^\alpha >0;
    \end{equation}
    moreover,
    \begin{equation*}
    \int_{\{ |z| > 1 \}} |z|^2 \,\nu(dz) < \infty.
    \end{equation*}
\end{itemize}}
We now consider the McKean-Vlasov equation \eqref{non-1} with the drift \eqref{e:dri} and the corresponding mean-field particle system defined by \eqref{e:mean} with initial values $X_0 \sim \mu_0$ and $X_0^{i,n} \sim \mu_0^n$, $1 \leq i \leq n$, respectively,
 which are i.i.d. random variables having finite second moments. Moreover,
  $(Z_t^i)_{t\ge0}$, $1\le i\le n$,
  are i.i.d.\ L\'evy processes with the same law as $(Z_t)_{t\ge0}$. Under
  the assumptions above, we can prove that both \eqref{non-1} and \eqref{e:mean} have unique strong solutions, which we denote by $(X_t)_{t \geq 0}$ and $(X_t^{i,n})_{t \geq 0}$,
   $1\le i\le n$, respectively;
  moreover, for any $t_0>0$ and $1\le i\le n$, $\sup_{0<t\le t_0}\Ee(|X_t|^2+|X_t^{i,n}|^2)<\infty$.
To see this, one can follow the standard arguments of
 \cite[Theorem 2.1]{Wang}, \cite[Theorem 2.1]{Gra} and \cite[Proposition 1.2]{JMW},
 see the proof of Proposition \ref{P:exist--} in Section \ref{section3} for details.
 Note that all the processes $(X_t^{i,n})_{t \geq 0}$, $1\le i\le n$, have the same marginal laws due to the uniqueness of solutions to the SDE \eqref{e:mean}.
   In the present paper, we prove the following result on uniform in time propagation of chaos for weakly interacting mean-field particle systems with L\'evy jumps.

\begin{theorem}\label{Main-4}
Suppose that assumptions $(1$-{\rm i}$)$--$(1$-{\rm iii}$)$ hold.
Denote by $(X_t)_{t \geq 0}$ and $(X_t^{i,n})_{t \geq 0}$,
 $1\le i\le n$,
 the unique strong solutions to \eqref{non-1} and \eqref{e:mean} with initial distributions $\mu_0$ and $\mu_0^n$,
 respectively. Suppose that $\mu_0$ and $\mu_0^n$ have finite first moments.
Let $\mu_t$ be the marginal law of $X_t$, and
$\mu_t^{n}$ the marginal law of $X^{i,n}_t$ for any $t > 0$ and any $1\le i\le n$.
    Then there exists a constant $K^*_{\tilde b_2}>0$ such that for any
    $K_{\tilde b_2}\in(0,K^*_{\tilde b_2}]$, $t>0$ and $n\ge 2$,
    $$W_1(\mu_t,\mu_{t}^n)\le C_0e^{-\lambda t} W_1(\mu_0,\mu_0^n)+C_0n^{-1/2},\quad t>0$$
holds for some positive constants $\lambda$ and $C_0$
independent of $t$ and $n$. \end{theorem}

\begin{remark}\label{remarkLevyMeasureAssumption}
    Condition \eqref{assumption:concentrationLevymeasure} is an assumption about sufficient concentration of the L\'{e}vy measure $\nu$ around zero
    (i.e., sufficient small jump activity of the L\'{e}vy process $(Z_t)_{t \geq 0}$). It is, however, a weak requirement that can be satisfied even for singular measures. For example,
    if $$\nu(dz)\ge \I_{\{0<z_1\le
        1\}}\frac{c}{|z|^{d+\alpha}}\,dz$$ holds for $z = (z_1, \ldots, z_d) \in \R^d$ and some constants $c>0$ and
    $\alpha\in (0,2)$, then \eqref{assumption:concentrationLevymeasure} is satisfied (see
    \cite[Example 1.2]{Lwcw}).
\end{remark}

In order to study convergence of solutions of McKean-Vlasov SDEs to their invariant measures, we need weaker assumptions than for proving the uniform in time propagation of chaos in Theorem \ref{Main-4}. Namely, \emph{we can consider \eqref{non-1} with a drift term $b: \R^d \times \mathscr{P}_1(\R^d) \to \R^d$ of general form.}
 Here $\mathscr{P}_1(\R^d)$ denotes the family of all probability measures on $\R^d$ with finite first moment, equipped with the topology of weak convergence metrized by $W_1$ (cf. \cite[Theorem 6.9]{Vil}). \emph{We impose the following continuity and contractivity at infinity conditions on the drift.
\begin{itemize}
    \item[(2-i)]
 $b(x,\mu)$ is continuous on $\R^d \times \mathscr{P}_1(\R^d)$
  in the product topology, and
 there exist constants $K_1$, $l_0\ge0$ and $K_2$, $K_3 > 0$ such that for any $x_1$, $x_2\in \R^d$
 with $x_1\neq x_2$ and any
  $\mu_1$, $\mu_2 \in
    \mathscr{P}_1(\R^d)$,
    \begin{equation}\label{assumptionOneSidedLipschitzForMcKean}\begin{split}\frac{\langle b(x_1,\mu_1)-b(x_2,\mu_2),x_1-x_2\rangle}{|x_1-x_2|} \le &  K_1|x_1-x_2|\I_{\{|x_1-x_2|\le l_0\}} \\ &-K_2|x_1-x_2|\I_{\{|x_1-x_2|>l_0\}}
    +K_3W_{1}(\mu_1,\mu_2).\end{split}  \end{equation}
    Furthermore, there is a constant $C_1>0$ such that for all $\mu \in \mathscr{P}_1(\R^d)$,
    \begin{equation}\label{assumptionLinearGrowthForMcKean}
    |b(0,\mu)|\le C_1\left(1+\int_{\R^d} |z|\,\mu(dz)\right).
    \end{equation}
 \end{itemize}
For the L\'{e}vy measure $\nu$, as in condition \eqref{assumption:concentrationLevymeasure}, we assume
\begin{itemize}
    \item[(2-ii)] For the function $J(r)$ given by \eqref{defJ}, there exists a constant $\alpha \in (0,1)$ such that
    \begin{equation}\label{assumption:concentrationLevymeasure2}
    \lim_{r\to 0}\inf_{s\in (0,r]} J(s)s^\alpha >0;
    \end{equation}
moreover,
    \begin{equation*}
    \int_{\{ |z| > 1 \}} |z|\, \nu(dz) < \infty.
    \end{equation*}
\end{itemize}}

We have the following result on convergence of
the solution of (\ref{non-1}) to a stationary distribution.

\begin{theorem}\label{Main-1} Suppose that assumptions $(2$-{\rm i}$)$--$(2$-{\rm ii}$)$ hold. Let $\mu_t$ denote the marginal law of a strong solution $(X_t)_{t\ge0}$ of \eqref{non-1} with initial distribution $\mu_0$
having finite first moment. Then there
    exists a constant $K^*_{3}>0$ such that for any
    $K_{3}\in(0,K^*_{3}]$, the associated McKean-Vlasov equation
    \eqref{non-1} has a unique invariant probability measure $\mu$ with
    finite first moment such that
    $$W_1(\mu_t,\mu)\le Ce^{-\lambda t},
    \quad t>0$$ for some positive constants $\lambda$ and $C$ independent of $t$.  \end{theorem}

Note that here, in contrast to condition (1-iii) required for Theorem \ref{Main-4}, we only assume that the L\'{e}vy measure $\nu$ has a finite first (and not necessarily second) moment. Taking into account Remark \ref{remarkLevyMeasureAssumption}, we see that our result covers McKean-Vlasov SDEs driven e.g.\ by nonsymmetric $\alpha$-stable processes with $\alpha \in (1,2)$, but also a much larger class of L\'{e}vy processes.
Note also that existence of an invariant measure in Theorem \ref{Main-1}
 essentially requires the constant $K_3$ in \eqref{assumptionOneSidedLipschitzForMcKean} to be sufficiently small,
and that condition \eqref{assumptionLinearGrowthForMcKean} is
only
used to ensure existence of a non-explosive solution to \eqref{non-1}.

Recently, Y. Song in \cite{Song} applied
Malliavin calculus to obtain exponential ergodicity for McKean-Vlasov equations with L\'{e}vy jumps
in the total variation distance, under an assumption similar to \eqref{assumptionOneSidedLipschitzForMcKean} but with $l_0 = 0$ and some additional regularity assumptions on the L\'{e}vy measure, see
 \cite[Theorem 1.5]{Song}.
 For the proof of Theorem \ref{Main-1} we use different methods, and we are able to obtain exponential ergodicity in the $L^1$-Wasserstein distance under relaxed assumptions, with an arbitrary $l_0 \geq 0$.

The assumption (2-i) on the drift $b(x,\mu)$ in Theorem \ref{Main-1}
can be further weakened, if we assume that $b(x,\mu)$ is of the special form \eqref{e:dri}.
Namely, \emph{we assume
\begin{itemize}
	\item[(3-i)]
	$b(x,\mu)$ satisfies \eqref{e:dri} and there exists a constant $\lambda > 0$ such that
	\begin{equation}\label{drift1-0000}\langle b_1(x) , x \rangle \leq - \lambda |x|^2\end{equation}
	for all $x \in \R^d$ with large enough $|x|$. Moreover, there are constants $K_1$, $K_2$ and $K_3>0$ such that for all $x_1,x_2\in \R^d$
with $x_1\neq x_2$ and $\mu_1,\mu_2\in \mathscr{P}_1(\R^d)$,
	\begin{equation}\label{e:3i2}
	\frac{\langle b_1(x_1)-b_1(x_2),x_1-x_2\rangle}{|x_1-x_2|} \le   K_1 |x_1-x_2|$$ and  $$\frac{\langle b_2(x_1,\mu_1)-b_2(x_2,\mu_2),x_1-x_2\rangle}{|x_1-x_2|} \le   K_2 |x_1-x_2| +K_3W_{1}(\mu_1,\mu_2).
	\end{equation}
	There is also a constant $B_0>0$ such that for all $x\in \R^d$ and $\mu\in\mathscr{P}_1(\R^d)$,
	\begin{equation}\label{e:3i3}
	|b_2(x,\mu)|\le B_0\left(1+\int |z|\,\mu(dz)+|x|\right).
	\end{equation}
\end{itemize}	
Furthermore, condition \eqref{assumption:concentrationLevymeasure2} on the concentration of the L\'{e}vy measure $\nu$ around zero can be only required to hold for a component of $\nu$.
\begin{itemize}
	\item[(3-ii)] We have $\int_{\{ |z| > 1 \}} |z| \,\nu(dz) < \infty$ and for any $\theta>0$, there exists a measure $0<\nu_\theta\le \nu$ such that ${\rm {supp}}\,\nu_\theta \subset B(0,1)$, $\int_{\{|z|\le 1\}} |z|\,\nu_\theta(dz)\le\theta,$ and
	\begin{equation}\label{ppllss}
	\lim_{r\to 0}\inf_{s\in (0,r]} J_{\nu_\theta}(s)s^\alpha >0,
	\end{equation} where $\alpha:=\alpha(\theta)\in (0,1)$ and $$J_{\nu_\theta}(s):= \inf_{x\in \R^d: |x|\le s} \big[\nu_\theta\wedge (\delta_x\ast \nu_\theta)\big]( \R^d)>0.$$
\end{itemize}}

Then we have the following result.

\begin{theorem}\label{Main-1add}
	Suppose that assumptions $(3$-{\rm i}$)$--$(3$-{\rm ii}$)$ hold. Let $\mu_t$ denote the marginal law of a strong solution $(X_t)_{t\ge0}$ of \eqref{non-1} with initial distribution $\mu_0$
having finite first moment.
Then there exist
positive constants $K_2^*, K_3^*, B_0^*$ such that for any $K_2\in (0,K_2^*]$, $K_3\in (0, K_3^*]$ and $B_0\in (0, B_0^*]$,
 the associated McKean-Vlasov equation \eqref{non-1} has a unique invariant probability measure $\mu$ with finite first moment such that
	$$W_1(\mu_t,\mu)\le Ce^{-\lambda t},
\quad t>0$$ for some positive constants $\lambda$ and $C$ independent of $t$.
\end{theorem}

At the end of this section, we give some
specific examples of drift terms that satisfy the corresponding assumptions.
 \begin{itemize}

\item[(1)] Each $b_1(x)$ below satisfies \eqref{cond:drif1} and \eqref{drift1-0000}.

\begin{itemize}
\item[(i)] Let $b_1(x)=-\lambda x+ U(x)$, where $U: \R^d\to \R^d$ is bounded.

\item[(ii)] Let $b_1(x)=-\nabla V(x)+U(x),$ where $V(x)=|x|^{2\beta}$ with $\beta>1$ and $U: \R^d\to \R^d$ is Lipschitz continuous, i.e., there is a constant $c>0$ such that for all $x,y\in \R^d$,
$|U(x)-U(y)|\le c|x-y|$.
\end{itemize}

\smallskip

It is obvious that both $b_1$ above satisfy \eqref{drift1-0000}. For $b_1$ given in (i), we have
$$\langle b_1(x)-b_1(y),x-y\rangle\le -\lambda |x-y|^2+\|U\|_\infty |x-y|,\quad x,y\in \R^d,$$ which implies that it fulfills \eqref{cond:drif1}. For $b_0(x)=-\nabla V(x)$ with $V(x)=|x|^{2\beta}$ for some $\beta>1$, it has been proven in \cite[Section 6, Example 1]{CG} that for any $x,y\in \R^d$,
$$\langle b_0(x)-b_0(y),x-y\rangle\le -2\beta 2^{3(1-\beta)}|x-y|^{2\beta}.$$ With this at hand, we can  see that $b_1$ given in (ii) also fulfills \eqref{cond:drif1}.

\smallskip

\smallskip

\item[(2)] For any $x\in \R^d$ and $\mu \in
    \mathscr{P}_1(\R^d)$,
        define $b_2(x,\mu)=\int_{\R^d} \bar{b}_2(x,z)\,\mu(dz)$, where $\bar{b}_2(x,z)$ satisfies that
$$
    |\bar{b}_2(x,z) - \bar{b}_2(y,z')| \leq K_{b_2}(|x-y|+|z-z'|)
$$
    holds for some constant $K_{b_2} > 0$ and for all $x$, $y$, $z$ and $z' \in \R^d$.
  Then we can verify that \eqref{e:3i2}
  and \eqref{e:3i3} also hold true; see Remark \ref{remarkDrift} for more details.
  \end{itemize}
Combining $b_1(x)$ and $b_2(x,\mu)$ above, one can obtain
multiple
examples of $b(x,\mu)$ which satisfy the corresponding assumptions in Theorems \ref{Main-4}, \ref{Main-1} and \ref{Main-1add}.

\ \

The remaining part of the paper is organized as follows. In Section \ref{suce1} we first
present a general
framework for studying
L\'{e}vy-driven SDEs with distribution independent drifts via the coupling operator. Then we
apply techniques based on couplings and Lyapunov functions to prove results on convergence to stationary distributions for such equations, including Theorem \ref{mainSDEresult}.
 In particular, in
  Subsection \ref{coupling-op}, we also introduce a new coupling for SDEs with distribution independent drifts by combining the refined basic coupling with the synchronous coupling, which is crucial
 in the
 study of the exponential ergodicity for the SDE \eqref{non-1} with a distribution-dependent drift.
  In Section \ref{section3} we first prove a slightly more general version of Theorem \ref{Main-1} and then we discuss its further extensions using the Lyapunov function -- based on methods of Section \ref{suce1}, thus proving Theorem \ref{Main-1add}. Finally, in Section \ref{section4} we prove Theorem \ref{Main-4}.

\section{Exponential convergence for SDEs with L\'evy noise}\label{suce1}
In this section, we will study the SDE \eqref{non-1} with distribution-independent drift, i.e.,
  \begin{equation}\label{s1}
  d X_t=b(X_{t})\,d t+ d Z_t,\quad X_0=x\in \R^d,
  \end{equation}
where $b: \R^d\rightarrow\R^d$ is a measurable function, and $Z=(Z_t)_{t\ge0}$ is a pure jump L\'{e}vy process on $\R^d$.
Denote by $\nu$ the L\'evy measure of the process $Z$. We always assume that
{\it the SDE \eqref{s1} has a unique strong solution}, which
is true, for example, for any noise $Z$ if the drift $b(x)$ is continuous and satisfies a one-sided Lipschitz condition, see \cite[Theorem 2]{GK}, or if $b(x)$ is H\"{o}lder continuous and $Z$ is a L\'evy process whose L\'evy measure satisfies some integrability conditions at zero and at infinity and whose transition semigroup enjoys certain regularity properties, see e.g.\ \cite{CSZ, kuehn-rs, Po1, Zhang1}.

\subsection{Coupling approach to exponential convergence: general framework} \label{coupling-op}
In this part, we present a general approach, based on the probabilistic coupling method, to study the exponential convergence of the SDE given by \eqref{s1}.
Let $X:=(X_t)_{t\geq 0}$ be the (unique) strong solution to the SDE \eqref{s1}. Then, the infinitesimal generator of $X$
acting on $C_b^2(\R^d)$
is given by
  \begin{equation}\label{SDE-generator}
  Lf(x)= \int\!\!\big(f(x+z)-f(x)-\langle\nabla f(x), z\rangle\I_{\{|z|\leq 1\}} \big)\,\nu(dz)+\langle b(x), \nabla f(x)\rangle.
  \end{equation}
To consider couplings of $X$, we will make use of the coupling operator of the generator $L$. A linear operator $\widetilde L: C_b^2(\R^{2d})\to B_b(\R^d)$ is called a coupling operator of the SDE given by \eqref{s1} (or of the generator $L$ given by \eqref{SDE-generator}), if
$$\widetilde L (f\otimes 1)(x,y)=Lf(x),\,\,\,\, \widetilde L(1\otimes g)(x,y)=Lg(y),\quad f,g\in C_b^2(\R^d),$$
where $(f\otimes g)(x,y)=f(x)g(x)$ denotes the tensor product of two functions $f$ and $g$. If the operator $\widetilde L$ can generate a Markov process $(X_t,Y_t)_{t\ge0}$ on $\R^{2d}$, then we call $(X_t,Y_t)_{t\ge0}$ a Markovian coupling process of the process $X$ determined by \eqref{s1}. In this case, $\widetilde L$ is called a Markovian coupling operator of $X$ (or $L$). In particular, let $(\widetilde P_t)_{t\ge0}$ be the semigroup of the Markovian coupling process $(X_t,Y_t)_{t\ge0}$, and $(P_t)_{t\ge0}$ be the semigroup of the process $X$. Then, for any $f,g\in C_b^2(\R^d)$ and $t>0$,
$$\widetilde P_t(f\otimes 1)(x,y)=P_tf(x),\,\,\,\, \widetilde P_t(1\otimes g)(x,y)=P_tg(y).$$

Let $\Phi$ be a function on $\R^d\times \R^d$
such that $\Phi(0,0)=0$ and $\Phi$ is strictly positive elsewhere. Given two probability measures $\mu_1$ and
$\mu_2$ on $\R^d$, we define the following quantity (which can be called
a Wasserstein-type distance or a Kantorovich distance)
$$
  W_\Phi(\mu_1,\mu_2)=\inf_{\Pi\in \mathscr{C}(\mu_1,\mu_2)}\int_{\R^d\times\R^d} \Phi(x,y)\,d\Pi(x,y),
$$
where $\mathscr{C}(\mu_1,\mu_2)$ is the collection of all measures on
$\R^d\times\R^d$ having $\mu_1$ and $\mu_2$ as marginals. In
particular, when $\Phi(x,y)=|x-y|$, $W_\Phi$ is just the
standard $L^1$-Wasserstein distance, which is simply denoted by
$W_1$ in the following; on the other hand, when $\Phi(x,y)=\I_{\{x\neq y\}}$, $W_\Phi$
leads to the total variation distance
$W_\Phi(\mu_1,\mu_2)=\frac{1}{2}\|\mu_1-\mu_2\|_{\var}.$

The
following statement provides a general tool for showing exponential
convergence in Wasserstein-type distances via the coupling method.

\begin{proposition}\label{p-w}
 Suppose that $\widetilde{L}$ is a coupling operator for the SDE given by \eqref{s1}, generating a non-explosive
 Markovian coupling process $(X_t,Y_t)_{t\ge0}$
 such
 that $X_t=Y_t$ for all $t\ge T$, where
 $T:=\inf\{t\ge0: X_t=Y_t\}$ is the coupling time of the process $(X_t,Y_t)_{t\ge0}$. Assume that there exist a constant $\lambda>0$ and a sequence of
 non-negative functions
 $\{\Phi_n(x,y)\}_{n\ge1}$ on $\R^{2d}$
such that
for any $n\ge1$,  $\Phi_n(x,x)=0$ for all $x\in \R^d$,
$\Phi_n(x,y)>0$ for all $x\neq y\in \R^d$, and
$\widetilde{L}\Phi_n(x,y)$ is
pointwise well defined, and
such that for  $n\ge 1$ large enough and for all $x,y\in \R^d$ with $1/n\le |x-y|\le n$,
\begin{equation}\label{e:proposition}
  \widetilde{L}\Phi_n
  (x,y)\le - \lambda \Phi_n(x,y).
\end{equation}
Then for any $t>0$ and $x,y\in\R^d$,
$$
  W_{\Phi_\infty}(
  P_t(x,\cdot), P_t(y,\cdot))\leq
  \Phi_\infty (x,y)e^{-\lambda t},
 $$
where  $\Phi_\infty=\liminf_{n\to\infty}\Phi_n$
and $P_t(x,\cdot)$ is the transition probability of the process $(X_t)_{t \geq 0}$ solving the SDE $\eqref{s1}$.
\end{proposition}

\begin{proof}
The proof follows step 2 of the proof of \cite[Theorem 3.1]{Lwcw} with some modifications. For the
sake of completeness, we present all the details here. Let
$(X_t,Y_t)_{t\ge0}$ be the
Markovian coupling process corresponding to the coupling operator $\widetilde L$
in the statement. To prove the desired assertion, it is enough to verify that for
$x,y\in\R^d$ with $|x-y|>0$ and any $t>0$,
  $$ \widetilde{\Ee}^{(x,y)}\Phi_\infty(X_t,Y_t)\leq \Phi_\infty(x,y)e^{-\lambda t},$$
where $\widetilde{\Ee}^{(x,y)}$ is the expectation of $(X_t,Y_t)_{t\ge0}$ starting from $(x,y)$.

For $n\geq 1$ define the stopping time
  $$T_n=\inf\{t>0: |X_t-Y_t|\notin [1/n, n]\}.$$
Since the coupling process $(X_t,Y_t)_{t\ge0}$ is non-explosive, we have $T_n\uparrow T$ a.s.\ as $n\to\infty$.
For any $x,$ $y\in\R^d$ with $|x-y|>0$, we take $n\ge 1$ large enough such that $1/n<|x-y|<n$.  Let $\{\Phi_n\}_{n\ge1}$ be the sequence of
non-negative functions, and $\lambda$ be the constant given in the statement. Then,  according to \eqref{e:proposition}, for any $m\ge n$,
  $$\aligned
  &\widetilde{\Ee}^{(x,y)}\big[e^{\lambda(t\wedge T_{n})} \Phi_m(X_{t\wedge T_{n}}, Y_{t\wedge T_{n}})\big]\\
  &=\Phi_m(x,y)+\widetilde{\Ee}^{(x,y)}\bigg(\int_0^{t\wedge T_{n}} e^{\lambda s}\big[\lambda\Phi_m(X_{s},Y_{s})+\widetilde{L} \Phi_m(X_{s},Y_{s})\big]\,d s\bigg)\\
  &\le \Phi_m(x,y).\endaligned$$
 Thus by Fatou's lemma, first letting $m\to \infty$ and then $n\to\infty$ in the above inequality, we find that
  $$ \widetilde{\Ee}^{(x,y)}\big(e^{\lambda(t\wedge T)}\Phi_\infty(X_{t\wedge T}, Y_{t\wedge T})\big)\leq \Phi_\infty(x,y). $$
Thanks to the facts that $Y_t=X_t$ for $t\geq T$ and $\Phi_\infty(x,x)=\liminf_{n\to\infty}\Phi_n(x,x)=0$ for all $x\in \R^d$, we have $\Phi_\infty(X_t,Y_t)=0$ for all $t\geq T$, which implies
  $$ \widetilde{\Ee}^{(x,y)}\big(e^{\lambda(t\wedge T)}\Phi_\infty(X_{t\wedge T},Y_{t\wedge T})\big)=e^{\lambda t} \widetilde{\Ee}^{(x,y)} \big(\Phi_\infty(X_t,Y_t)\I_{\{T>t\}}\big) = e^{\lambda t}\widetilde{\Ee}^{(x,y)} \Phi_\infty(X_t, Y_t).$$
This completes the proof.
\end{proof}

To apply Proposition \ref{p-w}, we need to construct a coupling operator $\widetilde L$ and a sequence of non-negative functions $\{\Phi_n(x,y)\}_{n\ge1}$ on $\R^{2d}$ which are strictly positive outside the diagonal such that \eqref{e:proposition} is satisfied. For the exponential convergence in terms of a total variation-type distance, we will take
$\Phi_n(x,y)\in C^2(\R^{2d})$ such that $\Phi_n(x,x)=0$ and
\begin{equation}\label{eq:Phinadd}
\Phi_n(x,y)=\psi_n(|x-y|)+\varepsilon(V(x)+V(y)),\quad |x-y|\ge 1/n,
\end{equation}
where $\varepsilon>0$, $V\in C^2(\R^d)$ is a Lyapunov function and $\{\psi_n(|x-y|)\}_{n\ge1}$ is a sequence of functions approximating $a+\psi(|x-y|)$ for some constant $a>0$ and a bounded concave function $\psi$ on $[0,\infty)$ with $\psi(0)=0$. It is obvious that $\lim_{n\to\infty}\Phi_n(x,y)$ is comparable to 
$(V(x)+V(y))\I_{\{x\neq y\}}$
 and so, with this choice of $\Phi_n$, Proposition \ref{p-w} could yield the exponential convergence in terms of the 
 $V$-variation
 norm given for any probability measures $\mu_1$ and $\mu_2$ as 
 $\|\mu_1-\mu_2\|_{\operatorname{Var}, V}:=\sup_{|f|\le V}|\mu_1(f)-\mu_2(f)|.$
The details will be provided in
Subsection \ref{subsection22}.
As pointed out at the beginning of \cite[Subsection 2.1]{EGZ}, the choice of those additive metrics is partially motivated by the paper \cite{HM}.
Concerning the exponential convergence in terms of $L^1$-Wasserstein-type distances, we will take
\begin{equation}\label{eq:Phinmult}
\Phi_n(x,y)=
\Phi(x,y):=\psi(|x-y|)(1+\varepsilon (V(x)+V(y)))
\end{equation}
for all $n\ge1$ and $x,y\in \R^d$, with a bounded concave function $\psi$ on $[0,\infty)$ satisfying $\psi(0)=0$ and a Lyapunov function $V\in C^2(\R^d)$. Note that $\lim_{|x-y|\to0}\Phi_n(x,y)=0$, while the function $\Phi_n$ given by \eqref{eq:Phinadd}, associated with the variation-type distance above, is discontinuous near the diagonal. The details of the construction involving \eqref{eq:Phinmult} will be provided in
Subsection \ref{Sec2.3}.

In order to construct the operator $\widetilde{L}$, we will adopt the refined basic coupling for the SDE \eqref{s1} from \cite{Lwcw}.
To describe the refined basic coupling, we start with a coupling for the L\'evy measure $\nu$.
To this end, we introduce the notation
$$x\to x+z,\quad \nu(dz)$$
for a transition from a point $x \in \R^d$ to the point $x + z$, with the jump intensity $\nu(dz)$.
 Roughly speaking, the essential idea of the basic coupling is to make the two marginal processes jump to the same point with the biggest possible rate, where the biggest jump rate is the maximal common part of the jump intensities. In the L\'evy setting, it takes the form
$$\mu_{y-x}(dz):=[\nu\wedge (\delta_{y-x} \ast\nu)](dz),$$ where $x$ and $y$ correspond to the positions of the two marginal processes before the jump.
Note that
for $x\neq0$,
\begin{equation}\label{ee:ffeerr}
    \mu_x(\R^d)
    \leq  \int_{\{|z|\le|x|/2\}}\,\left( \delta_x*\nu\right)(dz)+\int_{\{|z|>|x|/2\}}\,\nu(dz) \leq 2 \int_{\{|z|\ge|x|/2\}}\,\nu(dz)<\infty,
\end{equation}
i.e., $\mu_x$ is a finite measure on $(\R^d, \mathscr{B}(\R^d))$ for any $x\neq0$.
Let $\kappa_0$ be
a fixed constant. For any $x$, $y\in\R^d$ and $\kappa\in(0, \kappa_0]$, define
 $
  (x-y)_{\kappa}=\big(1\wedge \frac{\kappa}{|x-y|}\big)(x-y).
 $
We use the convention that $(x-x)_\kappa=0$. Then the refined basic coupling of $L$, first introduced in \cite[Section 2.1]{Lwcw}, is given as follows:
  \begin{equation}\label{basic-coup-3}
  (x,y)\longrightarrow
    \begin{cases}
    (x+z, y+z+(x-y)_\kappa), & \frac12 \mu_{(y-x)_\kappa}(dz),\\
    (x+z, y+z+(y-x)_\kappa), & \frac12 \mu_{(x-y)_\kappa}(dz),\\
    (x+z, y+z), & \big(\nu - \frac12 \mu_{(y-x)_\kappa}  -\frac12 \mu_{(x-y)_\kappa}\big)(dz).
    \end{cases}
  \end{equation}
We see that if $|x-y|\leq \kappa$, then \eqref{basic-coup-3} is reduced into
$$
  (x,y)\longrightarrow
    \begin{cases}
    (x+z, y+z+(x-y)), & \frac12 \mu_{y-x}(dz),\\
    (x+z, y+z+(y-x)), & \frac12 \mu_{x-y}(dz),\\
    (x+z, y+z), & \big(\nu - \frac12 \mu_{y-x}  -\frac12 \mu_{x-y}\big)(dz).
    \end{cases}
$$ The first row in the coupling above corresponds to the two marginal processes jumping to the same point. Note that the distance between the two marginals decreases from $|x-y|$ to $|(x+z)- (y+z+(x-y))|=0$; this is indeed the idea of the \emph{basic coupling} for Markov $q$-processes in \cite[Example 2.10]{Chen}.
However, unlike in the basic coupling, here we send the marginal processes to the same point only with half the maximal possible probability, since we want to apply the synchronous (rather than the independent) coupling to the remaining jumps (see \cite[Section 2.1]{Lwcw} for more details).
Hence we also need the second row, which
 corresponds to the change of the distance from $|x-y|$ to $2|x-y|$, and the last row, which is just a synchronous movement.
 If $|x-y|>\kappa$, then according to the first two rows in \eqref{basic-coup-3}, the distances after the jump are $|x-y|-\kappa$ and $|x-y|+\kappa$, respectively.
Hence the marginal processes can jump to the same point only if they are already close to each other before the jump (based on the threshold parameter $\kappa > 0$). Otherwise, they can only move slightly closer towards each other.
    Note that the introduction of $\kappa$ prevents a situation in which the two marginal processes could never couple if they only have finite range jumps (i.e., the sizes of jumps are bounded).
 Since this construction can be interpreted as a modification of the basic coupling from \cite[Example 2.10]{Chen}, we call the coupling given by \eqref{basic-coup-3} the \emph{refined basic coupling} for pure jump L\'evy processes.

 Motivated by the
 intuitive description above,
 we define for any $F\in C_b^2(\R^d\times \R^d)$,
\begin{equation}\label{e:couplg} \begin{split}\widetilde  L& F(x,y)\\
=&\frac{1}{2}\int\Big(F(x+z,y+z+(x-y)_\kappa)-F(x,y)-\langle\nabla_x F(x,y),z\rangle \I_{\{|z|\le1\}}\\
&\qquad\qquad -\langle\nabla_y F(x,y),z+(x-y)_\kappa\rangle\I_{\{|z+(x-y)_\kappa|\le 1\}}\Big)\,\mu_{(y-x)_\kappa}(dz)\\
&+\frac{1}{2}\int\Big(F(x+z,y+z+(y-x)_\kappa)-F(x,y)-\langle\nabla_x F(x,y),z\rangle\I_{\{|z|\le 1\}}\\
&\qquad\qquad -\langle\nabla_y F(x,y),z+(y-x)_\kappa\rangle\I_{\{|z+(y-x)_\kappa|\le 1\}}\Big)\,\mu_{(x-y)_\kappa}(dz)\\
&+\int\Big(F(x+z,y+z)-F(x,y)-\langle\nabla_x F(x,y), z\rangle\I_{\{|z|\le 1\}}\\
 &\qquad\qquad-\langle\nabla_y F(x,y), z\rangle\I_{\{|z|\le 1\}}\Big)\Big(\nu-\frac{1}{2}\mu_{(y-x)_\kappa}-\frac{1}{2}\mu_{(x-y)_\kappa}\Big)(dz)\\
&+\langle\nabla_x F(x,y),b(x)\rangle+\langle\nabla_y F(x,y), b(y)\rangle.\end{split}\end{equation}
Due to the facts that $\mu_x$ is a finite measure on $(\R^d,\mathscr{B}(\R^d))$ for all $x\neq 0$ and
$
 \delta_{x}* \mu_{-x}=\mu_{x}
$ (see \cite[Corollary
A.2]{Lwcw}), we can check that $\widetilde L$ defined by \eqref{e:couplg} is a coupling operator for the generator $L$ given by \eqref{SDE-generator}.

Below we briefly show that the coupling operator $\widetilde L$ generates a Markovian coupling process, which  
can be obtained as
the unique strong solution to an SDE on $\R^{2d}$.
 First, by the L\'evy-It\^{o}
decomposition, there exists a Poisson random measure $N$ associated
with $(Z_t)_{t\ge0}$ such that
$$dZ_t=\int_{\{|z|>1\}}z\, N(dt,dz)+\int_{\{|z|\le1\}} z\,\widetilde
N(dt,dz),$$ where $\widetilde N(dt,dz)=N(dt,dz)-dt\,\nu(dz)$ is the
compensated Poisson random measure.
Following the ideas from
\cite[Section 2.2]{Maj} and  \cite[Section 2.2]{Lwcw}, we extend the Poisson random measure $N$ from  $\R_+\times \R^d$ to $\R_+\times \R^d\times [0,1]$ in the following way \begin{equation*}\label{Poisson-meas}
  N(ds,dz,du) =\sum_{\{0<{s'}\le s, \Delta Z_{s'}\neq 0\}}\delta_{(s', \Delta Z_{s'})}(ds,dz)\I_{[0,1]}(du)
  \end{equation*}
and we write
  $$Z_t=\int_0^t\int_{\R^d\times[0,1]} z\,\bar{N}(ds,dz,du),$$
where
  $$\bar{N}(ds,dz,du)=\I_{\{|z|>1\}\times[0,1]}N(ds,dz,du)+\I_{\{|z|\le1\}\times [0,1]}\widetilde{N}(ds,dz,du).$$
We further define the control function $\rho$ as follows: for any $x,z\in\R^d$,
  $$\rho(x,z)
  =\frac{\mu_x(dz)}{\nu(dz)}
  =\frac{\nu \wedge (\delta_x\ast \nu) (dz)}{\nu(dz)}\in [0,1]. $$ Recall that for any $x\neq0$, $(x)_\kappa=(1\wedge (\kappa/|x|))x$.
 Fix any $x,y\in \R^d$ with $x\neq y$. We consider the system of equations:
  \begin{equation}\label{SDE-coup-eq-1}
  \begin{cases}
  dX_t=b(X_t)\, dt+dZ_t,& X_0=x,\\
  dY_t= b(Y_t)\,dt+ dZ_t+dL^\ast_t, & Y_0=y,
  \end{cases}
  \end{equation}
   where
  $$dL^\ast_t=\int_{\R^d\times [0,1]}
  S(U_{t-},z,u) \,{N}(dt,dz,du)$$with
 $U_t=X_t-Y_t$ and
  $$
  S(U_t,z,u)=(U_{t})_{\kappa} \Big[ \I_{\{u\le \frac12 \rho((-U_{t})_{\kappa},z)\}} - \I_{\{\frac12 \rho((-U_{t})_{\kappa},z)< u\le \frac12 [\rho((-U_{t})_{\kappa},z)+\rho((U_{t})_{\kappa},z)]\}}\Big]. $$
    According to \cite[Propositions 2.2 and 2.3]{Lwcw}, the SDE \eqref{SDE-coup-eq-1} has a unique strong solution, which is  a non-explosive coupling process $(X_t,Y_t)_{t\ge0}$ of the SDE \eqref{s1}.
    Since we assume that \eqref{s1} (i.e., the first equation
in \eqref{SDE-coup-eq-1}) has a non-explosive and pathwise unique strong solution $(X_t)_{t\ge0}$, the
sample paths of $(Y_t)_{t\ge0}$ can be obtained by repeatedly modifying those of the unique strong solution of the
following equation:
$$d\widetilde Y_t=b(\widetilde Y_t)\,dt + dZ_t,\quad Y_0=y$$ with aid of the so-called interlacing technique. Hence the uniqueness of the strong solution to \eqref{SDE-coup-eq-1} immediately follows; see the proof of \cite[Proposition 2.2]{Lwcw}.
  Note
  also
  that in this argument we do not need to assume that $b(x)$ is locally Lipschitz continuous if the SDE \eqref{s1} has a unique strong solution (which can be the case for non-locally Lipschitz drifts too, see e.g. \cite{CSZ, kuehn-rs, Po1, Zhang1}).
  Moreover, the generator of $(X_t,Y_t)_{t\ge0}$ is the refined basic coupling operator
  $\widetilde L$ defined by \eqref{e:couplg},
  and $X_t=Y_t$ for any $t\ge T,$ where $T=\inf\{t\ge0: X_t=Y_t\}$ is the coupling time of the process $(X_t,Y_t)_{t\ge0}$.
  In particular, the refined basic coupling $\widetilde L$ generates a Markovian coupling process $(X_t,Y_t)_{t\ge0}$ which satisfies the assumptions in Proposition \ref{p-w}.

 \smallskip

  For   the
  later use, we further show two properties for the coupling operator $
 \widetilde L$ and the coupling process $(X_t,Y_t)_{t\ge0}$. First, due to
 \eqref{e:couplg} and
 the fact (see \cite[Corollary
A.2]{Lwcw} again) that
$
 \delta_{x}* \mu_{-x}=\mu_{x},
$
it holds that, for any $\psi\in C_b^1([0,\infty))$ and $x$, $y\in\R^d$ with $x\neq y$,
  \begin{equation}\label{proofth2544}\allowdisplaybreaks\aligned
    &\widetilde{L} \psi(|x-y|) \\   &=\frac{1}{2} \mu_{(x-y)_{\kappa}}(\R^d)\Big[\psi\big(|x-y|+\kappa\wedge |x-y|\big) +  \psi\big(|x-y|-\kappa\wedge |x-y|\big)\\
    &\hskip85pt  - 2\psi(|x-y|)\Big] +\frac{\psi'(|x-y|)}{|x-y|}\langle b(x)-b(y),x-y\rangle.
  \endaligned
  \end{equation}
Second,
let $(X_t,Y_t)_{t\ge0}$ be the coupling process constructed above.
Recall that for any $t\ge0$, $U_t=X_t-Y_t$.
Then, it follows from the system \eqref{SDE-coup-eq-1} that
  \begin{equation*}
  \begin{split}
  dU_t &=(b(X_t)-b(Y_t))\, dt - \int_{\R^d\times [0,1]}  S(U_{t-},z,u)\, {N}(dt,dz,du).
  \end{split}
  \end{equation*}
Take $\psi\in C_b^1([0,\infty))$ with $\psi\ge0$. By the It\^{o} formula,
  \begin{equation}\label{e:coup-func}\begin{split}
  d\psi(|U_t|)
  =& \frac{\psi'(|U_t|)}{|U_t|} \<U_t, b(X_t)-b(Y_t)\> \, dt\\
  &+\int_{\R^d\times [0,1]} \big(\psi(|U_{t-}-S(U_{t-},z,u)|)-\psi(|U_{t-}|)\big)\, N(dt,dz,du)\\
  =& \frac{ \psi'(|U_t|)}{|U_t|}\langle U_t, b(X_t)-b(Y_t)\rangle\,dt\\
&+\int_{\R^d\times [0,1]}(\psi(|U_{t-}-S(U_{t-},z,u)|)-\psi(|U_{t-}|))\,\nu(dz)\,du\,dt\\
&+\int_{\R^d\times
[0,1]}(\psi(|U_{t-}-S(U_{t-},z,u)|)-\psi(|U_{t-}|))\,\widetilde{N}(dt,dz,du).
\end{split} \end{equation}
This along with \eqref{proofth2544} implies that we can
rewrite \eqref{e:coup-func} as follows
$$d\psi(|U_t|)=\widetilde{L}\psi(|U_t|)\,dt+ d M_t^\psi,$$ where
$$dM^\psi_t=\int_{\R^d\times
[0,1]}(\psi(|U_{t-}-S(U_{t-},z,u)|)-\psi(|U_{t-}|))\,\widetilde{N}(dt,dz,du)$$ is a local martingale.
We also note that \eqref{proofth2544} can be deduced from \eqref{e:coup-func} directly. In the sequel the coupling operator for the L\'{e}vy process $Z$ without a drift (i.e., with $b(x) = 0$ for all $x\in\R^d$ in \eqref{s1}) will be denoted by $\widetilde{L}_Z$.

\smallskip

In order to deal with the SDE \eqref{non-1} with distribution-dependent drift in Section \ref{section3},
we will
need to
introduce a new coupling for \eqref{s1} by combining
the refined basic coupling
with the synchronous coupling.
To this end, for any fixed $\delta > 0$, let $\phi_\delta : [0,\infty) \to [0,1]$ be a smooth function such that
\begin{equation}\label{e:testfunction}\phi_\delta(r)=\begin{cases} = 0,\quad & 0\le r \le \delta/2,\\
\in [0,1],\quad & \delta/2\le r\le \delta,\\
= 1,\quad & r\ge \delta.\end{cases} \end{equation} We then define a coupling operator $\widetilde{L}_Z^{\delta}$ given for any $F\in C_b^2(\R^d\times \R^d)$ by
\begin{equation}\label{truncatedCoupling}
\begin{split}
\widetilde{L}_Z^{\delta}F(x,y)
=&
\phi_\delta(|x-y|)\cdot\widetilde{L}_Z F(x,y)+(1-\phi_\delta(|x-y|))\cdot\widetilde{L}_{Z,*} F(x,y),
\end{split}
\end{equation} where $\widetilde{L}_{Z,*}$ denotes the synchronous coupling operator for the L\'{e}vy process $Z$ with the L\'{e}vy measure $\nu$, i.e., for any $F\in C_b^2(\R^d\times \R^d)$, \begin{align*} \widetilde{L}_{Z,*} F(x,y)=\int\Big(F(x+z,y+z)-F(x,y)&-\langle\nabla_x F(x,y), z\rangle\I_{\{|z|\le 1\}}\\
&-\langle\nabla_y F(x,y), z\rangle\I_{\{|z|\le 1\}}\Big)\,\nu(dz).\end{align*}
Similarly as
for $\widetilde{L}_Z$, we can prove that $\widetilde{L}_Z^{\delta}$ defined this way is indeed a coupling operator for
the L\'{e}vy process $Z$, cf.\  \cite[Section 2]{Lwcw}.
The coupling works as the refined basic coupling when the distance between the marginal processes is larger than $\delta$,
is the synchronous coupling when the distance between the marginal processes is smaller than $\delta/2$, and is a mixture of the refined basic coupling and the synchronous coupling in the remaining case. We can construct a coupling process $(X_t,Y_t^{\delta})_{t \geq 0}$
with the generator $\widetilde{L}^\delta$ defined for $F \in C_b^2(\R^d\times \R^d)$ as
\begin{equation*}
	\widetilde{L}^\delta F(x,y) = \widetilde{L}_Z^\delta F(x,y) + \langle \nabla_x F(x,y) , b(x) \rangle + \langle \nabla_y F(x,y) , b(y) \rangle
\end{equation*}	
 exactly as described above by considering the system of SDEs
\begin{equation*}
\begin{cases}
dX_t=b(X_t)\, dt+dZ_t,& X_0=x,\\
dY_t^{\delta}= b(Y_t^{\delta})\,dt+ dZ_t+dL^{\ast,\delta}_t, & Y_0^{\delta}=y,
\end{cases}
\end{equation*}
where
\begin{equation}\label{Ldelta}
dL^{\ast,\delta}_t=\int_{\R^d\times [0,1]}S^{\delta}(U^\delta_{t-}, z,u) \,{N}(dt,dz,du)
\end{equation}
with
\begin{equation}\label{Vdelta}\begin{split}
S^\delta(U^\delta_{t},z,u)=&(U_{t}^{\delta})_{\kappa}\Big[ \I_{\{u\le \frac12 \rho((-U_{t}^{\delta})_{\kappa},z)\}}\\
 & \qquad\quad - \I_{\{\frac12 \rho((-U_{t}^{\delta})_{\kappa},z)< u\le \frac12 [\rho((-U_{t}^{\delta})_{\kappa},z)+\rho((U_{t}^{\delta})_{\kappa},z)]\}}\Big]
 \phi_\delta(|U_t^\delta|)\end{split}
\end{equation}
and  $U_t^{\delta} = X_t - Y_t^{\delta}$.

\  \

In the next two subsections, we will apply the refined basic coupling operator $\widetilde L$ to study the exponential ergodicity for the SDE \eqref{s1}.
Hence,
according to Proposition \ref{p-w}, the main task is to construct
an appropriate
sequence of non-negative functions $\{\Phi_n(x,y)\}_{n\ge1}$ such that \eqref{e:proposition} is satisfied. By the remarks below the proof of Proposition \ref{p-w},
in
Subsections \ref{subsection22} and \ref{Sec2.3}
we will consider convergence in terms of additive distance and multiplicative distance, respectively.

\subsection{Exponential ergodicity for the SDE \eqref{s1}: additive distance}\label{subsection22}

In this part, we assume the following conditions.

{\noindent{\bf Assumption (A)}\it
\begin{itemize}
\item[(i)] There is a constant $K_1>0$ such that for all $x,y\in
\R^d$,
$$\langle b(x)-b(y),x-y\rangle\le K_1(|x-y|^2\vee |x-y|).$$
\item[(ii)]
    There is a constant $0<\kappa_0\le1$ such that
  \begin{equation}\label{th10}
  J(\kappa_0):=\inf_{x\in \R^d:\, |x|\le \kappa_0} \big[\nu\wedge (\delta_x\ast \nu)\big]( \R^d)>0.
  \end{equation}
\item[(iii)] There are a $C^2$-function $V:\R^d\to [1,\infty)$ and constants $C, \lambda>0$ such that $V(x)\to \infty$ as $|x|\to \infty$, and
\begin{equation}\label{e:ly} LV(x)\le C-\lambda V(x),\quad x\in \R^d, \end{equation}
where $L$ is the infinitesimal generator given by \eqref{SDE-generator}.
\end{itemize}}

We emphasise that in this section we always assume that the SDE \eqref{s1} has a unique strong solution. Note that condition (i) is weaker than the standard one-side Lipschitz condition (see Assumption {\bf (B)} (i) below).
In particular, it is easy to see that condition (i) is satisfied
 for any bounded measurable function $b$.
Note that in the one-dimensional case the SDE \eqref{s1} driven by a symmetric $\alpha$-stable process with $\alpha\in(1,2)$ has a strong solution, even when $b(x)$ is only bounded measurable; see \cite{TTW}.
However, in the general case
some additional properties of the drift and the driving noise may be required in order to guarantee existence of a unique strong solution to \eqref{s1}, as discussed at the beginning of this section. Moreover, in order to construct a Lyapunov function $V$ satisfying \eqref{e:ly} we may have to impose
some further assumptions on $b$ and the driving L\'{e}vy process $Z$. For example, if
$\int_{\{|z|\ge1\}}|z|\,\nu(dz)<\infty$ and there are constants $c_0$, $l_0>0$ such that $\langle b(x) , x \rangle \leq - c_0|x|^2$ for all $x\in \R^d$ with $|x|\ge l_0$,
then \eqref{e:ly} holds with a radial function $V \in C^2(\R^d)$ such that
 $V(x)\ge1$ for all $x\in \R^d$ and
 $V(x) = 1+|x|$ for $|x| \ge 1$,
see Lemma \ref{lem:var2} and Remark \ref{remarkLyapunovForStandardSDE} for details.
Alternatively, suppose that $\int_{\{|z|\ge1\}}|z|\,\nu(dz)<\infty$ does not hold, but we assume the weaker condition $\int_{\{|z|\ge1\}}|z|^{\alpha}\,\nu(dz)<\infty$ for some $\alpha \in (0,1)$. In this case, if there are constants $c_0$, $l_0>0$ such that $\langle b(x) , x \rangle \leq - c_0|x|^2$ for all $x\in \R^d$ with $|x|\ge l_0$, then \eqref{e:ly} holds with $V(x) = (1+|x|^2)^{\beta/2}$ for any $\beta \in (0,\alpha]$, cf.\ (the proof of) \cite[Lemma 3.1]{HMW}.

\begin{theorem}\label{th1:le} Under Assumption {\bf (A)},
there are constants $C_0,\lambda_0>0$ such that for all $x,y\in \R^d$ and $t>0$,
$$W_{\Phi}(P_t(x,\cdot), P_t(y,\cdot))\le C_0 e^{-\lambda_0t} \Phi(x,y),$$ where $\Phi(x,y)=(V(x)+V(y))\I_{\{x\neq y\}}.$
 \end{theorem}

As a direct consequence of Theorem \ref{th1:le}, we have

\begin{corollary}\label{co1:le} Under Assumption {\bf (A)}, the process $(X_t)_{t\ge0}$ is exponentially ergodic; more explicitly, there are a unique invariant probability measure $\mu$, a constant $\lambda_0$ and a measurable function $C_0(x)$ such that for all $x\in \R^d$ and $t>0$,
$$\|P_t(x,\cdot)-\mu\|_{\operatorname{Var}, V}\le C_0(x) e^{-\lambda_0t},$$ where for any probability measure $\mu_1$ and $\mu_2$,
$$\|\mu_1-\mu_2\|_{\operatorname{Var}, V}=\sup_{|f|\le V}|\mu_1(f)-\mu_2(f)|.$$  \end{corollary}
\begin{proof}Assumption {\bf(A)} (iii) is the well-known Foster-Lyapunov type condition in the study of stability of Markov processes, see \cite{MT} for more details. In particular,
according to \eqref{e:ly} and \cite[Theorem 3.1]{MT}, we know that
$\Ee^xV(X_t)<\infty$ for all $x\in \R^d$ and $t>0$. On the other
hand, it was proven in \cite[Lemma 2.1]{HM} that for any probability
measures $\mu_1$ and $\mu_2$,
$$\|\mu_1-\mu_2\|_{\operatorname{Var}, V}=W_{\Phi}(\mu_1, \mu_2),$$ where $\Phi(x,y)=(V(x)+V(y))\I_{\{x\neq
y\}}.$ Then, combining these two conclusions with Theorem
\ref{th1:le} and some standard arguments (see for example the proof
of \cite[Corollary 1.8]{Maj}), we can prove the desired
assertion.\end{proof}

\begin{remark}
We give some comments on Assumption {\bf(A)} and Theorem \ref{th1:le}.
 \begin{itemize}
\item[
(i)] We first make remarks on Assumption {\bf (A)}(i) and (ii) respectively. When
$b(x)$ is locally bounded, (i) holds locally, i.e., (i) holds for
all $x,y\in B(0,R)$ and $R>0$,
cf.\ Remark \ref{R:2.6} on how to extend Theorem \ref{th1:le} to hold under a local version of (i).
Condition (ii) is very weak and can be true even for finite L\'evy
measures; see \cite[Proposition 1.5]{SW2} and \cite[Remark 1.7]{Maj}.

\item[
(ii)] Theorem \ref{th1:le} and Corollary \ref{co1:le}
give us a new way to yield the exponential ergodicity for SDEs with
additive L\'evy noises. We emphasize that the common approach to
ergodicity is based on verifying the irreducibility and the strong
Feller property of the associated Markov processes; however, we
believe that such approach is not easy to apply only under
Assumption {\bf (A)}.
We further note that an analogous result for diffusions was shown in \cite[Theorem 2.1]{EGZ} under
an additional
growth condition \cite[Assumption 2.3]{EGZ} for the Lyapunov function $V$ (
which essentially states that
there is a constant $C_0>0$ such that $V(x)\ge C_0(1+|x|)$ for large $|x|$, see \cite[Lemma 2.1]{EGZ}); nevertheless,
in the present setting we do not require this additional assumption.

\item[
(iii)] Explicit estimate for the (rate) constant $\lambda_0$ in Theorem \ref{th1:le} is available at the end of its proof. \end{itemize}
\end{remark}

To prove Theorem \ref{th1:le}, we begin with the following lemma.

\begin{lemma}\label{lem-test-funct}  Let $\psi\in C^1([0,\infty))$.
Then the following hold.
\begin{itemize}
\item[
(i)] If $\psi'$ is decreasing, then for any $0\le \delta\le r$, $$ \psi(r+\delta)+\psi(r-\delta)-2\psi(r)\leq 0.$$
\item[
(ii)] Suppose that $\psi\in C([0,2l_0])\cap C^4((0,2l_0])$ for some $l_0>0$ such that $\psi'> 0$, $\psi''\leq 0,$  $\psi'''\geq 0$ and $\psi^{(4)}\leq 0$ on $(0,2l_0]$. Then, for any  $0\leq \delta \leq r\leq l_0$,
  $$\psi(r+\delta) +\psi(r-\delta)-2\psi(r)\leq \psi''(r)\delta^2.$$
\end{itemize}
\end{lemma}

\begin{proof} The assertion
 (i) is trivial if $\delta=0$, thus we assume $\delta>0$.
 By the mean value formula, there exist constants $\xi_1\in (r,r+\delta)$ and $\xi_2\in (r-\delta, r)$ such that
  $$\psi(r+\delta)-\psi(r)=\psi'(\xi_1)\delta$$
and
  $$\psi(r-\delta)-\psi(r)= -\psi'(\xi_2)\delta.$$
Therefore,
  $$\psi(r+\delta)+\psi(r-\delta)-2\psi(r)=(\psi'(\xi_1)-\psi'(\xi_2))\delta\le 0,$$
since $\psi'$ is decreasing.

To prove
(ii), we will still assume $\delta>0$. Similar to the proof of
(i), by the Taylor formula, there exist constants $\xi_1\in (r,r+\delta)$ and $\xi_2\in (r-\delta, r)$ such that
  \begin{align*}
  \psi(r+\delta)&=\psi(r)+ \psi'(r)\delta +\frac12 \psi''(r)\delta^2 +\frac16 \psi'''(\xi_1)\delta^3,\\
  \psi(r-\delta)&=\psi(r)- \psi'(r)\delta +\frac12 \psi''(r)\delta^2 -\frac16 \psi'''(\xi_2)\delta^3.
  \end{align*}
Therefore,
  $$\psi(r+\delta) +\psi(r-\delta)-2\psi(r)= \psi''(r)\delta^2 +\frac{\delta^3}6 \big[\psi'''(\xi_1)- \psi'''(\xi_2)\big] \leq \psi''(r)\delta^2,$$
since $\psi'''$ is decreasing due to $\psi^{(4)}\le 0$.
\end{proof}

\begin{proof}[Proof of Theorem $\ref{th1:le}$]
(i) For any $n\ge1$, define $\Phi_n(x,y)\in C^2(\R^{2d})$ such that
$$\Phi_n(x,y)\begin{cases}=\psi(|x-y|),&\quad 0\le |x-y|\le 1/(n+1),\\
\le a+\psi(|x-y|)+\varepsilon(V(x)+V(y)),&\quad 1/(n+1)\le |x-y|\le 1/n,\\
= a+\psi(|x-y|)+\varepsilon(V(x)+V(y)),&\quad |x-y|\ge 1/n. \end{cases}$$ Here,
$$\psi(r)=\begin{cases}1-e^{-cr},&\quad r\in (0,2l_0],\\
1-e^{-2cl_0}+ce^{-2cl_0}\frac{r-2l_0}{1+r-2l_0},&\quad r>2l_0\end{cases}$$ with
$$l_0=\sup_{(x,y)\in S_0}|x-y|+1,\quad  c=\frac{4K_1l_0}{J(\kappa) \kappa^2}+1$$ and  $$S_0=\left\{(x,y)\in \R^{2d}: \lambda (V(x)+V(y))\le 16 C\right\};$$ the function $V$ and the constants $\lambda, C, K_1,J(\kappa)$ are given in Assumption {\bf(A)}; and
$$a=\frac{4K_1c}{J(\kappa)}+\kappa^2c^2e^{-cl_0},\quad \varepsilon=\frac{1}{16 C} J(\kappa)\kappa^2 c^2 e^{-cl_0}.$$
Note that $\sup_{(x,y)\in S_0}|x-y| < \infty$
 and so $l_0<\infty$,
 since $V(x) \to \infty$ as $|x| \to \infty$ by
 Assumption {\bf(A)}(iii).

According to
the definition of $\Phi_n(x,y)$, we know that for any $n\ge1$ and $x,y\in \R^d$ with $|x-y|\ge 1/n$,
\begin{equation}\label{e:llkkss}\widetilde L \Phi_n(x,y)
\le  \widetilde L \psi_n(|x-y|)+\widetilde  L (\varepsilon(V(x)+V(y))),\end{equation} where
$$\psi_n(r)\begin{cases}=\psi(r),&\quad 0\le r\le 1/(n+1),\\
\le a+\psi(r),&\quad 1/(n+1)\le r\le 1/n,\\
= a+\psi(r),&\quad r\ge 1/n. \end{cases}$$

(ii)
For any $n\ge 1$ and $r\in [1/n,\infty)$, we have $\psi_n(r)=a+\psi(r)$ and $\psi_n'(r)=\psi'(r)$. Therefore, for any $\kappa \in (0,\kappa_0]$ and $r>0$,
\begin{align*} \psi_n(r-\kappa\wedge r)&=\psi_n(r-\kappa \wedge r)\I_{\{r>\kappa\}}\le (a+\psi(r-\kappa\wedge r))\I_{\{r>\kappa\}}\\
&=(a+\psi(r-\kappa\wedge r))- (a+\psi(r-\kappa\wedge r))\I_{\{r\le \kappa\}}\\
&= (a+\psi(r-\kappa\wedge r))- a\I_{\{r\le \kappa\}}.\end{align*}
Recall that, under Assumption {\bf(A)}(ii), for any $0<\kappa\le \kappa_0\le 1$, we have $J(\kappa)=\inf_{0<s\le
\kappa} J(s)>0$.
For
any $r\in (1/n,l_0]$,
\begin{equation}\label{e:ste1}\begin{split}&\frac{1}{2}J(\kappa \wedge r)\left[\psi_n(r+r\wedge \kappa)+\psi_n(r-r\wedge \kappa)-2\psi_n(r)\right]\\
&\le \frac{1}{2}J(\kappa \wedge r)\left[\psi(r+r\wedge \kappa)+\psi(r-r\wedge \kappa)-2\psi(r)\right]-\frac{a}{2}J(\kappa \wedge r)\I_{\{r\le \kappa\wedge l_0\}}.\end{split}\end{equation}
On the other hand, we can check that the function $\psi$ above satisfies all the conditions in both statements of Lemma \ref{lem-test-funct}.
Hence, according to \eqref{proofth2544}, \eqref{e:ste1} and Assumption {\bf(A)}(i),  for any $x,y\in \R^d$ with $1/n\le |x-y|\le \kappa,$
\begin{align*}\widetilde L \psi_n(|x-y|)\le &\frac{1}{2}J(|x-y|)\left[\psi_n(2|x-y|)-2\psi_n(|x-y|)\right]+K_1\psi_n'(|x-y|)\\
 \le& -\frac{a}{2} J(\kappa)+  \frac{1}{2}J(\kappa)|x-y|^2\psi''(|x-y|)+ K_1\psi_n'(|x-y|)\\
 \le& -\frac{a}{2} J(\kappa)+K_1c\le -\frac{a}{4} J(\kappa);\end{align*} for any $x,y\in \R^d$ with $\kappa<|x-y|\le l_0$,
\begin{align*}\widetilde L \psi_n(|x-y|)\le & \frac{1}{2}J(\kappa)\kappa^2\psi''(|x-y|)+K_1l_0\psi'(|x-y|)\\
\le& -\frac{1}{4}J(\kappa)\kappa^2 c^2 e^{-c|x-y|}\le -\frac{1}{4}J(\kappa)\kappa^2 c^2 e^{-cl_0};\end{align*}
and for any $x,y\in \R^d$ with $|x-y|> l_0$,
\begin{equation}\label{l:kk3}\begin{split}\widetilde L \psi_n(|x-y|)\le & K_1|x-y|\psi'(|x-y|)\\
 \le& K_1\max\Big\{\sup_{l_0\le r\le 2l_0} r\psi'(r), \psi'(2l_0)\sup_{r>2l_0}  r(1+r-2l_0)^{-2}\Big\}\\
=&K_1 l_0\psi'(l_0)\le\frac{1}{4}J(\kappa)\kappa^2 c^2 e^{-cl_0}.\end{split}\end{equation} Putting all the estimates together, we find that
\begin{align*}\widetilde L \psi_n(|x-y|)\le \begin{cases}-\frac{a}{4} J(\kappa),\quad &1/n\le |x-y|\le \kappa,\\
-\frac{1}{4}J(\kappa)\kappa^2 c^2 e^{-cl_0},\quad &\kappa<|x-y|\le l_0,\\
\frac{1}{4}J(\kappa)\kappa^2 c^2 e^{-cl_0},\quad &|x-y|>l_0,\end{cases} \end{align*}
On the other hand, by \eqref{e:ly} in Assumption {\bf(A)}(iii)
and the definition of the coupling operator, for all $x,y\in \R^d$,
\begin{align*}\widetilde L [\varepsilon (V(x)+V(y))]
=\varepsilon (LV(x)+LV(y))\le \begin{cases} 2\varepsilon C,\quad & (x,y)\in S_0,\\
-\frac{\varepsilon \lambda}{2}(V(x)+V(y)),\quad &
(x,y)\notin S_0.\end{cases}  \end{align*}
According to the choice of $\varepsilon$ and \eqref{e:llkkss}, we
conclude that
for any $n\ge1$ and $x,y\in \R^d$ with $|x-y|\ge 1/n$,
\begin{align*} \widetilde L
\Phi_n(x,y) \le & \widetilde L\psi_n(|x-y|)+\widetilde L(\varepsilon (V(x)+V(y)))\\
\le&\begin{cases}-\frac{1}{8}J(\kappa)\kappa^2 c^2 e^{-cl_0},\quad & (x,y)\in S_0,\\
-\frac{\varepsilon \lambda}{4}(V(x)+V(y)),\quad &
(x,y)\notin S_0\end{cases}\\
\le&-\lambda_0 [\psi_n(|x-y|)+\varepsilon (V(x)+V(y))]\\
=&-\lambda_0 \Phi_n(x,y),\end{align*}where $$\lambda_0=\min\left\{ \frac{\lambda J(\kappa)\kappa^2c^2 e^{-cl_0}}{8(\lambda+a\lambda+16C\varepsilon)}, \frac{4\varepsilon\lambda C}{16C\varepsilon+\lambda(1+a+ce^{-2cl_0})}\right\}.$$
This, along with Proposition \ref{p-w} and the fact that for all $n\ge1$ and $x,y\in \R^d$,
$$ \varepsilon (V(x)+V(y))\le  \psi_n(|x-y|)+\varepsilon (V(x)+V(y))\le [(1+a + ce^{-2cl_0})+\varepsilon] (V(x)+V(y)),$$ where we used $V(x) \geq 1$ for all $x\in \R^d$ and $\psi(r) \leq 1 + ce^{-2cl_0}$ for all $r \geq 0$, proves the desired assertion.
\end{proof}

\begin{remark}\label{R:2.6} Motivated by \cite[Section 1.6.1]{Zi}, we can extend Theorem \ref{th1:le} by replacing the global condition, Assumption {\bf (A)}(i) on $b(x)$, by a local one. More precisely,
in the proof above, we replace $K_1$ by $K_1(l_0)$ which satisfies that for all $x,y\in \R^d$ with $|x-y|\le l_0$,
\begin{equation}\label{ppp-1}\langle b(x)-b(y),x-y\rangle\le K_1(l_0)(|x-y|^2\vee |x-y|),\end{equation} and change the argument for \eqref{l:kk3} as that for any $x,y\in \R^d$ with $|x-y|> l_0$,
\begin{align*}\widetilde L \psi_n(|x-y|)\le & |b(x)-b(y)|\psi'(|x-y|)\\
 \le&|b(x)-b(y)|\max\Big\{\sup_{l_0\le r\le 2l_0}  \psi'(r), \psi'(2l_0)\sup_{r>2l_0}  (1+r-2l_0)^{-2}\Big\}\\
 = &|b(x)-b(y)|\psi'(l_0).\end{align*} Thus, we can follow the proof of Theorem \ref{th1:le} and obtain that, if for any $(x,y)
 \notin S_0$,
\begin{equation}\label{ppp-2}V(x)+V(y)\ge \frac{64C}{\lambda (J(\kappa)\kappa^2+4K_1(l_0)l_0)}|b(x)-b(y)|,\end{equation} then, there exists $\lambda_0>0$ (which is different from that in the proof of Theorem \ref{th1:le}) such that for any $n\ge1$ and $x,y\in \R^d$ with $|x-y|\ge 1/n$,
 $$\widetilde L \Phi_n(x,y)\le - \lambda_0 \Phi_n(x,y).$$ In particular, the conclusion of Theorem \ref{th1:le} still holds true under Assumption {\bf(A)} (ii)-(iii), \eqref{ppp-1} and \eqref{ppp-2}. \end{remark}

Finally, we present
the proof of Theorem \ref{mainSDEresult}.

\begin{proof}[Proof of Theorem $\ref{mainSDEresult}$]
It is easy to see that condition (\ref{drift3}) implies Assumption {\bf(A)}(i). Moreover, it can be shown that condition (\ref{drift1}),
 along with the assumption that $\int_{\{|z|\ge1\}}|z|\,\nu(dz)<\infty$,
 implies Assumption {\bf(A)}(iii) with a Lyapunov function $V$ such that $V(x) \geq 1$ for all $x \in \R^d$ and $V(x) = |x| + 1$ for $|x| \geq 1$, cf.\ Lemma \ref{lem:var2} and Remark \ref{remarkLyapunovForStandardSDE}. In particular, $V(x) \geq |x|$ for all $x \in \R^d$ and hence the weighted total variation distance $\| \cdot \|_{\operatorname{Var},V}$ from Corollary \ref{co1:le} dominates both the standard total variation and $L^1$-Wasserstein distances (see e.g. \cite[Theorem 6.15]{Vil} and \cite[Remark 2.3]{EGZ}), which finishes the proof.
\end{proof}

\subsection{Exponential ergodicity for the SDE \eqref{s1}: multiplicative distance}\label{Sec2.3} In this subsection,
we present an alternative way of studying exponential ergodicity for SDEs given by (\ref{s1}), via multiplicative Wasserstein pseudo-distances. The techniques that we will introduce here will also play a crucial role in Section \ref{section3} for studying exponential ergodicity of McKean-Vlasov equations and proving Theorem \ref{Main-1add}. We remark that for McKean-Vlasov SDEs we were not able to prove an analogue of Theorem \ref{th1:le} from the previous section, but only an analogue of a weaker Theorem \ref{th2:le} presented below.

We will work under the following set of conditions.

{\noindent{\bf Assumption (B)}\it
\begin{itemize}
\item[(i)] There is a constant $K_1>0$ such that for all $x,y\in
\R^d$,
$$\langle b(x)-b(y),x-y\rangle\le K_1|x-y|^2.$$
\item[(ii)] For any $\theta>0$, there exists a measure $0<\nu_\theta\le \nu$ such that ${\rm {supp}}\,\nu_\theta \subset B(0,1)$, $\int_{\{|z|\le 1\}} |z|\,\nu_\theta(dz)\le\theta,$ and
  \begin{equation}\label{th13}
  \lim_{r\to 0}\inf_{s\in (0,r]} J_{\nu_\theta}(s)s^\alpha >0,
  \end{equation} where $\alpha:=\alpha(\theta)\in (0,1)$ and $$J_{\nu_\theta}(s):= \inf_{x\in \R^d: |x|\le s} \big[\nu_\theta\wedge (\delta_x\ast \nu_\theta)\big]( \R^d)>0.$$
\item[(iii)] There exists a $C^2$-function $V:\R^d\to [1,\infty)$ such that the Lyapunov condition \eqref{e:ly} holds and
\begin{equation}\label{mmkk}\sup_{z\in B(x,2)}|\nabla V(z)|\le C_0 V(x),\quad x\in \R^d,\end{equation} where $C_0>0$ is a constant independent of $x\in \R^d$.
\end{itemize}}

Assumption \eqref{mmkk} is
a
growth condition on the Lyapunov function $V$.
It is satisfied, e.g., if $1\le V\in C^2(\R^d)$ and $V(x)=|x|^\alpha$ with $\alpha>0$ or $V(x)=\exp(|x|^\beta)$ with $\beta\in (0,1]$ for large $|x|$. We note that, in the context of
related results for
diffusions, similar conditions (see \cite[Assumptions 2.4 and 2.5]{EGZ}) are also imposed.
Similarly as in Subsection \ref{subsection22}, we remark that under
additional assumptions that $\int_{\{ |z| > 1 \}} |z|\, \nu(dz) < \infty$ and that there are constants
$c_0, l_0>0$ such that $\langle b(x) , x \rangle \leq -
c_0|x|^2$ for all $x \in \R^d$ with
 $|x|\ge l_0$. Assumption {\bf(B)} (iii) is satisfied with a radial function $V \in C^2(\R^d)$ such that
  $V(x)\ge1$ for all $x\in \R^d$ and
 $V(x)=1+|x|$ for all $x\in \R^d$ with $|x|\ge 1$; see Lemma \ref{lem:var2} and Remark \ref{remarkLyapunovForStandardSDE}.

\begin{theorem}\label{th2:le} Under Assumption {\bf (B)},
there are constants $\lambda_0$ and $C_0 > 0$ such that for all $x,y\in \R^d$ and $t>0$,
$$W_{\Phi}(P_t(x,\cdot),P_t(y,\cdot))\le C_0 e^{-\lambda_0t} \Phi(x,y),$$ where
$\Phi(x,y)=(|x-y|\wedge1)(V(x)+V(y)).$ \end{theorem}

The use of
the multiplicative distance $W_\Phi$ is inspired by the weak Harris theorem
introduced in \cite{HMS}, see also \cite[Section 2.2]{EGZ}. As mentioned in \cite{HMS}, the distance
of multiplicative form is more applicable for SDEs with degenerate
noises or infinite dimensional SDEs, where convergence in terms of
the total variation norm (and so the additive metric in the previous
subsection) does not hold. We note that the multiplicative distance
$W_\Phi$ indeed is only a multiplicative semimetric, since the
triangle inequality may be not true. See \cite[Section 4]{HMS} for
more details.
As shown in the proof of Theorem \ref{th2:le} below, we take the reference function corresponding to the multiplicative distance $W_\Phi$ of the form
$\psi(|x-y|)(1+\varepsilon(V(x)+V(y)))$, where $\psi(|x-y|)$ is a bounded concave function and is comparable to $|x-y|$ for all $x,y\in \R^d$ with $|x-y|\le 1$.
Note also that, since the function $\Phi(x,y)$ in Theorem \ref{th2:le} satisfies that $\Phi(x,y)\to0$ as $x\to y$, we can get from Theorem \ref{th2:le} that the associated semigroup of the process $(X_t)_{t\ge0}$ is Feller, i.e., for any $t>0$ and $f\in C_b(\R^d)$, $P_tf\in C_b(\R^d)$; see the proof of \cite[Proposition 1.5]{Lwcw}. However, such
an
assertion can not been deduced from Theorem \ref{th1:le}.

Similar to Corollary \ref{co1:le}, we have the following statement.

\begin{corollary}\label{co2:le}Suppose that $\int_{\{|z|\ge1\}}|z|\,\nu(dz)<\infty$ and Assumption {\bf(B)}{\rm(iii)} is satisfied for $V$ with $V(x)\ge c_0|x|$ for all $|x|$ large enough and some constant $c_0\in (0,1)$. Under Assumption {\bf (B)} {\rm(i)} and {\rm(ii)}, the process $(X_t)_{t\ge0}$ is
exponentially ergodic in terms of $W_1$-distance; more explicitly,
there are a unique invariant probability measure $\mu$ with finite
first moment, a constant $\lambda_0$ and a measurable function
$C_0(x)$ such that for all $x\in \R^d$ and $t>0$,
$$W_1(P_t(x,\cdot),\mu)\le C_0(x) e^{-\lambda_0t}.$$\end{corollary}
\begin{proof} Note that here, unlike in Corollary \ref{co1:le}, $W_\Phi$ is only a semimetric, as mentioned above. First,
it can be verified that the condition $\int_{\{|z|\ge1\}}|z|\,\nu(dz)<\infty$
along with Assumption {\bf(B)}(i) yields that $\Ee^x|X_t|<\infty$
for all $x\in \R^d$ and $t>0$. Hence, by Theorem
\ref{th2:le} and the fact that
$c_1W_1(\mu_1,\mu_2)\le W_\Phi(\mu_1,\mu_2)$ (which is implied by $c_1|x-y|\le \Phi(x,y)$ for all $x,y\in \R^d$, due to the definition of $\Phi(x,y)$ and our assumption that $V(x)\ge c_0|x|$ for all $|x|$ large enough and some constant $c_0\in (0,1))$ one can obtain the
existence of a unique invariant probability measure. With aid of this point, we can follow the argument of
Corollary \ref{co1:le} to prove the desired assertion.
\end{proof}

We now pass to the proof of Theorem \ref{th2:le}. We begin with the
following lemma.
\begin{lemma}\label{L:kk1}
Let $g\in C([0,2l_0])\cap C^3((0,2l_0])$ be such that
$g'(r)\geq 0,$ $ g''(r)\leq 0$ and $ g'''(r)\geq 0$ for any $r\in (0,2l_0].$
Then for all $c_1>0$ the function
  $$
  \psi(r):=\psi_{c_1}(r)= \begin{cases}
  c_1 r+ \int_0^r e^{-g(s)}\, ds ,& r\in [0, 2l_0],\\
 \psi(2l_0)+ \psi'(2l_0) \frac{r-2l_0}{1+r-2l_0}, & r\in(2l_0,\infty)
  \end{cases}
 $$
satisfies
\begin{itemize}
\item[
(i)] $\psi\in C^1([0,\infty))$ and $c_1r\le \psi(r)\le (c_1+1)r$ on $[0,2l_0]$;
\item[
(ii)]$\psi'> 0$, $\psi''\leq 0,$  $\psi'''\geq 0$ and $\psi^{(4)}\leq 0$ on $(0,2l_0]$;
\item[
(iii)] for any $0\le \delta\le r$, $$ \psi(r+\delta)+\psi(r-\delta)-2\psi(r)\leq 0;$$
\item[
(iv)] for any  $0\leq \delta \leq r\leq l_0$,
  $$\psi(r+\delta) +\psi(r-\delta)-2\psi(r)\leq \psi''(r)\delta^2.$$
\end{itemize}
\end{lemma}
Note that the points
(i) and
(ii) are easy to check, whereas
(iii) and
(iv) follow from Lemma \ref{lem-test-funct}.

Next, we present the
proof of Theorem \ref{th2:le}.

\begin{proof}[Proof of Theorem $\ref{th2:le}$]
 Under Assumption {\bf(B)}(ii), we will apply the refined basic coupling for the component $\nu_\theta$ of the L\'evy measure $\nu$, and couple the remaining mass synchronously, where $\theta>0$ is determined later. Let $\psi$ be the function given in Lemma \ref{L:kk1},  where
$$g(r)={C_*(2K_1+1)l_0^{2-\alpha}}{r^{\alpha}},\quad l_0=\sup_{(x,y)\in S_0}|x-y|+2L_1,\quad  c_1=e^{-g(2l_0)}$$ and  $$S_0=\left\{(x,y)\in \R^{2d}: \lambda (V(x)+V(y))\le L_2 C\right\}.$$ Here, $\alpha\in (0,1)$ is given in Assumption {\bf(B)}(ii), and $C_*, L_1$ and $L_2$ are large enough (their exact values will be determined later).

By the definition of $g$, we know that $g'(r)\ge0$, $g''(r)\leq 0$ and $g'''(r)\geq 0$ for any $r\in (0,2l_0]$. In particular, we can use Lemma \ref{L:kk1}.
On the other hand, by \eqref{th13}, there are constants $\kappa:=\kappa(\theta)\in (0,\kappa_0]$ where $\kappa_0 \leq 1$ is given by \eqref{th10} and $C_*:=C_*(\kappa,\theta)>0$ (both are independent of $l_0$) such that for all $r\in (0,2l_0]$, \begin{equation}\label{l:kk2}\sigma(r):=(\alpha C_*)^{-1}l_0^{\alpha-2} r^{1-\alpha}\le \frac{1}{2r} J(\kappa\wedge r)(\kappa\wedge r)^2.\end{equation} In particular,
$g'(r)=\frac{2K_1+1}{\sigma(r)}.$

In the following, let $F(x,y)=\psi(|x-y|)(1+\varepsilon(V(x)+V(y)))$ for any $x,y\in \R^d$, where $\varepsilon>0$ is determined later. For any $x\in \R^d$, set $\mu_{\theta,x}=\nu_\theta\wedge (\delta_x\ast \nu_\theta)$. Then, by \eqref{e:couplg} and some elementary calculations, we can find that for any $x,y\in \R^d$,
\begin{equation}\label{e:estim1}\begin{split}\widetilde L & F(x,y)\\
=&\widetilde L \psi(|x-y|) \cdot (1+\varepsilon(V(x)+V(y))) + \psi(|x-y|)\cdot \varepsilon \widetilde L(V(x)+V(y))\\
&+\frac{1}{2}\varepsilon(\psi(|x-y-(x-y)_\kappa|)-\psi(|x-y|))\\
&\quad\times\int [(V(x+z)-V(x))+(V(y+z+(x-y)_\kappa)-V(y))]\,\mu_{\theta,(y-x)_\kappa}(dz)\\
&+\frac{1}{2}\varepsilon(\psi(|x-y+(x-y)_\kappa|)-\psi(|x-y|))\\
&\quad\times\int [(V(x+z)-V(x))+(V(y+z-(x-y)_\kappa)-V(y))]\,\mu_{\theta,(x-y)_\kappa}(dz).\end{split}\end{equation}

For any $x,y\in \R^d$ with $|x-y|\le l_0$, by \eqref{proofth2544}, Lemma \ref{L:kk1}(iii), \eqref{l:kk2} and \eqref{e:ly},
\begin{align*}&\widetilde L F(x,y)\\
& \le\bigg(\!\!\!-\frac{1}{2} J(|x-y|\wedge \kappa)(|x-y|\wedge\kappa)^2e^{-g(|x-y|)} g'(|x-y|) +K_1|x-y|(c_1+e^{-g(|x-y|)})\! \bigg)\\
&\qquad\times (1+\varepsilon(V(x)+V(y)))\\
&\quad -\varepsilon \lambda(V(x)+V(y))\cdot \psi(|x-y|)+2C\varepsilon \cdot \psi(|x-y|)\\
&\quad +\varepsilon \psi( |x-y|) \\
&\qquad\times\bigg(\sup_{z\in B(x,1)}\nabla V(z) \int |z|\,\nu_\theta(dz) +\sup_{z\in B(y,2)}\nabla V(z)\int |z+(x-y)_\kappa|\,\mu_{\theta,(y-x)_\kappa}(dz)\bigg)\\
&\quad +\frac{1}{2}\varepsilon(\psi(2|x-y|)+\psi(|x-y|))\\
&\qquad \times\bigg(\sup_{z\in B(x,1)}\nabla V(z) \int |z|\,\nu_\theta(dz) +\sup_{z\in B(y,2)}\nabla V(z)\int |z+(y-x)_\kappa|\,\mu_{\theta,(x-y)_\kappa}(dz)\bigg)\\
&\le \bigg(\!\!\!-\frac{1}{2} J(|x-y|\wedge \kappa)(|x-y|\wedge\kappa)^2e^{-g(|x-y|)}\, \frac{2K_1+1}{\sigma(|x-y|)} +K_1|x-y|(c_1+e^{-g(|x-y|)})\!\bigg)\\
&\qquad\times  (1+\varepsilon(V(x)+V(y)))\\
&\quad -\varepsilon \lambda(V(x)+V(y))\cdot \psi(|x-y|)+2C\varepsilon \cdot \psi(|x-y|)\\
&\quad +2C_0\varepsilon(\psi(2|x-y|)+\psi(|x-y|)) (V(x)+V(y))\int |z|\,\nu_\theta(dz)\\
 &\le -|x-y| e^{-g(|x-y|)} \cdot (1+\varepsilon(V(x)+V(y)))\\
&\quad -\varepsilon \lambda(V(x)+V(y))\cdot \psi(|x-y|)+2C\varepsilon \cdot \psi(|x-y|)\\
&\quad +C_1\theta\varepsilon \psi(|x-y|)(V(x)+V(y)),\end{align*}
where the first inequality follows from the mean value theorem, in
the second inequality we used Assumption {\bf(B)}(iii), and in the
last inequality we used the facts that $\psi(2r)\le 2\psi(r)$ for
$0<r\le l_0$, $\int_{\{|z|\le 1\}}|z|\,\nu_\theta(dz)\le \theta$ and
Assumption {\bf(B)}(iii) together with the fact that $c_1 \leq
e^{-g(|x-y|)}$ for $|x-y| \leq l_0$. Note that $C_1$ is a constant
independent of $\theta$, $\varepsilon$, $\lambda$ and $l_0$, and
that the argument for the estimates of the last two terms in
\eqref{e:estim1} works for all $x,y\in \R^d$. Now let us choose
$\theta>0$ small enough so that $C_1\theta\le \lambda/4$ and take
$\varepsilon>0$ small enough so that $2C\varepsilon \psi(r)\le r
e^{-g(r)}$ for all $0<r\le l_0$. More precisely, we can take
$\varepsilon>0$ such that
\begin{equation}\label{e:kkffllff1}2C\varepsilon= \inf_{0\le r\le
l_0}r e^{-g(r)}\psi(r)^{-1}.\end{equation} Therefore, for all
$x,y\in \R^d$ with $|x-y|\le l_0$,
\begin{equation}\label{e:proofsss}\widetilde L F(x,y)\le -(\varepsilon \lambda/2) \psi(|x-y|)(V(x)+V(y)).\end{equation}
For any $(x,y)\notin S_0$, following the argument above and \eqref{l:kk3}, we can get that
\begin{equation}\label{e:thm26aux3}
\begin{split}
\widetilde L F(x,y) \le &4K_1l_0 e^{-g(l_0)}\I_{\{|x-y|\ge l_0\}}
 \cdot (1+\varepsilon(V(x)+V(y)))\\
&-(\varepsilon \lambda/2)(V(x)+V(y))\cdot \psi(|x-y|)\\
&+C_1\theta\varepsilon \psi(|x-y|)(V(x)+V(y))\\
\le&4K_1l_0 e^{-g(l_0)}\I_{\{|x-y|\ge l_0\}}\cdot (1+\varepsilon(V(x)+V(y)))\\
&-(\varepsilon \lambda/4)(V(x)+V(y))\cdot \psi(|x-y|),
\end{split}
\end{equation} where in the first inequality we used the definition of $S_0$ and the last inequality follows from the choice of $\theta$ above. Next, we choose $L_1\ge1$ large enough so that
$$4K_1e^{-C_*(1+2K_1)r^2}\le \lambda e^{-C_*(1+2K_1)}r^{-2/\alpha}/16,\quad r\ge L_1$$ and hence, since $l_0 > L_1$, we have
$$4K_1l_0e^{-g(l_0)}\le  \lambda e^{-C_*(1+2K_1)}l_0^{(\alpha-2)/\alpha}/16.$$ In particular,
\begin{equation}\label{e:thm26aux1}4K_1l_0e^{-g(l_0)}\le  \frac{\lambda}{16}\int_0^{l_0^{-(2-\alpha)/\alpha}} e^{-C_*(2K_1+1)l_0^{2-\alpha}s^\alpha}\,ds\le \frac{\lambda}{16} \psi(l_0).\end{equation}
Furthermore, we choose
\begin{equation*}
L_2 \geq \max \left( 4 , \frac{64 K_1 l_0 e^{-g(l_0)}}{\varepsilon C \psi(l_0)} \right) \,.
\end{equation*}
Then for all $x$, $y \in \R^d$ with $|x-y| > l_0$ we have
\begin{equation}\label{e:thm26aux2}
4K_1 l_0 e^{-g(l_0)} \leq \frac{\lambda \varepsilon}{16} \psi(|x-y|) (V(x) + V(y))
\end{equation}
(note that $|x-y| > l_0$ implies $(x,y) \notin S_0$). Hence, combining \eqref{e:thm26aux1} with \eqref{e:thm26aux2} and \eqref{e:thm26aux3},
 we see that for any $(x,y)\notin S_0$,
 \begin{equation}\label{e:thm26aux4}
 \widetilde L F(x,y)\le -(\varepsilon \lambda/8)(V(x)+V(y))\cdot \psi(|x-y|).
 \end{equation} This along with \eqref{e:proofsss} shows that \eqref{e:thm26aux4} holds for all $x,y\in \R^d$,
 which proves the desired assertion by Proposition \ref{p-w}.
\end{proof}

\begin{remark}
   The results discussed in the present paper
can also be obtained by applying other coupling operators. For instance, one could apply the coupling studied in \cite{Maj,Maj2} to obtain inequalities such as \eqref{e:proposition} for SDEs driven by L\'{e}vy processes with rotationally symmetric L\'{e}vy measures
 which are not required to satisfy the concentration around zero property \eqref{assumption:concentrationLevymeasure}, cf.\
  \cite[Remark 1.6]{Maj}. This would allow us to prove e.g. an analogue of Theorem \ref{th2:le} under different (neither strictly weaker nor stronger) assumptions on the noise.
      Yet another possibility would be to use the coupling from \cite{LW14, JWang}. See \cite{LSW} for a discussion on different couplings for L\'{e}vy processes and L\'{e}vy-driven SDEs. In the present paper we choose to work with the refined basic coupling given by \eqref{e:couplg} since it
      can apply to a very large class of non-symmetric L\'{e}vy measures.
\end{remark}

\begin{remark}  We present
two further remarks
on the assumptions on the drift in
Theorems \ref{th1:le} and
\ref{th2:le}. \begin{itemize}
\item[(i)]
Assumption {\bf(B)} (i) in Theorem \ref{th2:le} is the well known one-sided Lipschitz
condition on the drift. Note that it is stronger than the assumption we needed in Theorem \ref{th1:le}. By carefully checking the proof above, we
can see that Theorem \ref{th2:le} still holds true if Assumption
{\bf(B)} (i) is weakened into the following condition: {\it there are constants $K_1>0$ and
$\beta\in (0,1-\alpha)$ with $\alpha\in (0,1)$ in Assumption
{\bf(B)} $(ii)$ such that for all $x,y\in \R^d$,
$$\langle b(x)-b(y),x-y\rangle\le K_1 (|x-y|^{1+\beta}\vee
|x-y|^2).$$} For the modification of the proof to adjust to this weaker
assumption, one can refer to the proof of \cite[Theorem 4.2]{Lwcw}.

\item[(ii)]
Following the argument of
\cite[Section 1.6.2]{Zi}, one can also prove
a more general version of
Theorem \ref{th2:le} by replacing
the
one-sided Lipschitz
condition on the drift in Assumption {\bf(B)} by a local
counterpart,
and by imposing some growth condition on $\psi V$
similar to
\eqref{ppp-2}. The details are omitted here.
 \end{itemize}
\end{remark}

\section{Exponential convergence for McKean-Vlasov SDEs with L\'evy noise}\label{section3}
In this section, we are concerned with the McKean-Vlasov (distribution dependent) SDE with
jumps given by \eqref{non-1}. Let $\mathscr{P}_1(\R^d)$ be the set of probability measures on
$\R^d$ with finite first moment. Throughout this section, we always assume that
the following conditions hold.

{\noindent{\bf Assumption (H)}\it
\begin{itemize}
\item[(i)]
The
drift term $b(x,\mu)$ is continuous on $\R^d \times \mathscr{P}_1(\R^d)$, and that it satisfies the following one-sided Lipschitz condition:
\begin{equation}\label{e:assdrift}\frac{\langle b(x_1,\mu_1)-b(x_2,\mu_2),x_1-x_2\rangle}{|x_1-x_2|}
\le K(|x_1-x_2|+W_{1}(\mu_1,\mu_2))\end{equation} for all $x_1,x_2\in \R^d$
with $x_1\neq x_2$ and all
$\mu_1,\mu_2\in \mathscr{P}_1(\R^d)$
for some constant $K>0$,
as well as the following growth condition: there is a constant $C_1>0$ such that for all $\mu \in \mathscr{P}_1(\R^d)$,
\begin{equation}\label{e:assdrift00}|b(0,\mu)|\le C_1\left(1+\int_{\R^d} |z|\,\mu(dz)\right).\end{equation}
\item[(ii)] The L\'evy measure $\nu$ of the L\'evy process $Z:=(Z_t)_{t\ge0}$ satisfies
$$\int_{\{|z|\ge1\}}|z|\,\nu(dz)<\infty.$$
\end{itemize}}

We start with the following statement.

\begin{proposition}\label{P:exist--}Under Assumption {\bf (H)}, the McKean-Vlasov SDE $\eqref{non-1}$
has a unique non-explosive strong solution $(X_t)_{t\ge0}$ such that
$\Ee |X_t|<\infty$ for all $t>0$.   \end{proposition}

\begin{proof} This follows from (the proofs of)
\cite[Theorem 2.1]{Gra} and \cite[Theorem
2.1]{Wang}, and we sketch the
argument
here. Equation \eqref{non-1} is a special case of the McKean-Vlasov stochastic differential equation given in \cite[(2.1)]{Gra} with
$$\sigma(x,\mu)\equiv 0, \quad f(x,\mu,u)=\tilde f(x,\mu,u)=u$$ and $N(du,dt)$ being the Poisson random measure associated with the L\'evy process $(Z_t)_{t\ge0}$, with the
intensity
measure $\nu(du)$ on $(\R^d,\mathscr{B}(\R^d))$. Then, according to (the proof of) \cite[Theorem 2.1]{Gra} and Assumption {\bf(H)}, there exists a unique strong
solution for equation \eqref{non-1}. Note that the proof of \cite[Theorem 2.1]{Gra} is based on the parametrized approach by considering a sequence of SDEs $(X_t^{(n)})_{t>0}$ given by
$$dX_t^{(n)}=b(X_t^{(n)}, \mu^{(n-1)}_t)\,dt +dZ_t,\quad t\ge0, n\ge 1$$ with $X_t^{(0)}=x$ for all $t\ge0$ and $\mu^{(n-1)}_t$ being the law of $X_t^{(n-1)}$. This idea is the same as that for \cite[Theorem 2.1]{Wang}. In particular, according to the argument of
\cite[Theorem 2.1]{Wang},
the Lipschitz condition
$$|b(x_1,\mu_1)-b(x_2,\mu_2)|
\le K(|x_1-x_2|+W_{1}(\mu_1,\mu_2))$$ assumed in \cite[Theorem 2.1]{Gra} can be replaced by Assumption {\bf(H)}.
Moreover, one can show that the unique strong solution $(X_t)_{t\ge0}$ to the SDE \eqref{non-1} satisfies $\Ee |X_t|<\infty$ for all $t>0$, by following again the proof of \cite[Theorem 2.1]{Wang}. \end{proof}

The
existence and uniqueness of a strong solution
to the SDE \eqref{non-1} clearly implies the existence of a weak solution to \eqref{non-1}; see \cite[Definition 1.1]{Wang}.
Let $\mu_{X_t}$ be the distribution of the time marginal $X_t$ of the process $(X_t)_{t\ge0}$ with initial distribution $\mu_{X_0}$.
Then, it follows from \cite[Propositions 1.6 and 1.7]{JMW}
 that for any $f\in C^2_b(\R^d)$,
$$\left\{f(X_t)-f(X_0)-\int_0^t L[\mu_{X_s}]f(X_s)\,ds,t\ge0\right\}$$ is a $\Pp$-martingale, where \begin{equation}\label{e:inf1}\begin{split}L[\mu] f(x)=& \langle b(x,\mu), \nabla f(x)\rangle\\
&+\int(f(x+z)-f(x)-\langle\nabla f(x), z\rangle\I_{\{|z|\le
    1\}})\,\nu(dz).\end{split}\end{equation} That is, $L[\mu]f$ can be interpreted as the infinitesimal generator of the process $(X_t)_{t\ge0}$.

\smallskip

In the rest of this section, we will always assume that Assumption {\bf(H)} holds.

\subsection{Convergence in $W_1$: the contractivity at infinity approach}
 We say that the drift term
$b(x,\mu)$ in \eqref{non-1} satisfies ${\mathbf B(K_1, K_2,
l_0; K_3)}$, if \eqref{assumptionOneSidedLipschitzForMcKean} holds, i.e., for any $x_1,x_2\in \R^d$
with $x_1\neq x_2$ and $\mu_1, \mu_2 \in
\mathscr{P}_1(\R^d)$,
\begin{equation}\label{e:contractivityAtInfinity}
\begin{split}
\frac{\langle b(x_1,\mu_1)-b(x_2,\mu_2),x_1-x_2\rangle}{|x_1-x_2|} \le &  K_1|x_1-x_2|\I_{\{|x_1-x_2|\le l_0\}}\\
&-K_2|x_1-x_2|\I_{\{|x_1-x_2|>l_0\}}+K_3W_{1}(\mu_1,\mu_2),
\end{split}
\end{equation} where
$l_0\in[0,\infty)$, $K_1$, $K_2$ and $K_3\ge0$.

\begin{theorem}\label{thtpw}
Suppose that Assumption {\bf(H)} holds. Assume that
\begin{equation}\label{th10--}
  J(\kappa_0):=\inf_{x\in \R^d:\, |x|\le \kappa_0} \big[\nu\wedge (\delta_x\ast \nu)\big]( \R^d)>0
\end{equation} for some $0<\kappa_0\le1$,
and that the drift $b(x,\mu)$ satisfies
${\mathbf
B(K_1, K_2, l_0; K_3)}$. Suppose that there exists a nondecreasing and concave function $\sigma\in C([0,2l_0])\cap C^2((0,2l_0])$ such that  $g_1(r):=\int_{0}^r \frac{1}{\sigma(s)}\,ds$ is well
defined for all $r\in [0,2l_0]$, and for some $\kappa\in (0,\kappa_0]$,
  \begin{equation}\label{e:sigma}
  \sigma(r)\leq \frac1{2r} J(\kappa\wedge r) (\kappa\wedge r)^2, \quad r\in (0, 2l_0],
  \end{equation}
where $J(r)$ is defined by \eqref{defJ}. Let $\mu_{X_t}$ $($resp.\ $\mu_{Y_t}$$)$ be the distribution of the time marginal $X_t$ $($resp.\ $Y_t$$)$ of a solution to \eqref{non-1} with initial distribution $\mu_{X_0}$ $($resp. $\mu_{Y_0}$$)$.
Then for any $t>0$,
 $$  W_1(\mu_{X_t}, \mu_{Y_t})\le C e^{-\lambda t} W_1(\mu_{X_0}, \mu_{Y_0}),$$
  where $$\lambda=\frac{c_1c_2}{1+c_1}-\frac{(1+c_1)K_3}{2c_1},\quad C=\frac{1+c_1}{2c_1}$$ with
$c_2=(2K_2)\wedge g_1(2l_0)^{-1}$, $c_1=e^{-c_2 g(2l_0)}$ and
  $g(r)=(1+\frac{2K_1}{c_2}) g_1(r)$ for all $ r\in [0,2l_0].$

\end{theorem}

\begin{remark}\label{r:thtpw} We make some comments on Theorem \ref{thtpw} and its proof.
 \begin{itemize}
\item[
(i)]Theorem \ref{thtpw} extends \cite[Theorem 4.2]{Lwcw}, where the
drift term is distribution independent. As pointed out in
\cite[Remark 4.3(1)]{Lwcw}, when $b(x,\mu)=b(x)$ satisfies the
uniformly dissipative condition in the sense that there is a
constant $K_2>0$ such that for any $x,y\in \R^d$,
$$\langle b(x)-b(y),x-y\rangle\le -K_2|x-y|^2;$$ (that is, $l_0=0$ and $K_3=0$ in Theorem \ref{thtpw}),
then the constant $\lambda$ in the statement is equal to $K_2$, which is optimal.

\item[
(ii)] When $K_3$ is small enough so that $\lambda>0$, using Theorem \ref{thtpw} and the fact that
$\Ee|X_t|<\infty$ for all $t>0$, we can obtain
exponential ergodicity of $(X_t)_{t\ge0}$ in terms of
$W_1$-distance, e.g. see the proof of
\cite[Theorem 3.1(2)]{Wang}

\item[
(iii)]
    Note that for proving convergence of solutions to McKean-Vlasov SDEs we cannot use the classical argument involving a coupling $(X_t, Y_t)_{t \geq 0}$ such that $Y_t = X_t$ for all $t \geq T$, where $T$ is the coupling time, see \cite[Remarks on pages 598--599]{Wang}
or \cite[Remarks between Assumption 2.7 and Theorem 2.4]{EGZ}. Note also that for any stopping time $\tau$, if $(X_t)_{t \geq 0}$ is a solution to \eqref{non-1}, then the stopped process $(Y_t)_{t\ge0}$ defined by $Y_t := X_{t \wedge \tau}$ for all $t\ge 0$ solves
    \begin{equation*}
    Y_t = Y_0 + \int_0^{t \wedge \tau} b(Y_s, \mu_s) \,ds + Z_{t \wedge \tau},\quad t\ge0 \,,
    \end{equation*}
    which is not the same SDE since in general $\operatorname{Law}(Y_t) \neq \operatorname{Law}(X_t) = \mu_t$, see  \cite[Remarks on page 3]{HSS}. This means that applying the standard argument via Proposition \ref{p-w} cannot give us the exponential convergence for $W_{\Phi_\infty}(\mu_{X_t},\mu_{Y_t})$ as desired, see
    its proof for more details. Hence to consider convergence of McKean-Vlasov SDEs we will use a different approach, which is based on the combination of the refined basic coupling and the synchronous coupling  defined by \eqref{truncatedCoupling}.

\item[
(iv)] We further remark that the typical way of constructing couplings of L\'{e}vy-driven SDEs by using
the interlacing technique (see e.g.\ \cite[Proposition 2.2]{Lwcw} or
\cite[Section 2.4]{Maj}) is non-applicable here due to the lack of the strong Markov property of the solution $(X_t)_{t \geq 0}$, see the discussions in \cite{Gra}. Hence we apply here a different approach based on the results on non-linear martingale problems from \cite{JMW}.
\end{itemize}
\end{remark}

\begin{proof}[Proof of Theorem $\ref{thtpw}$] (i)
 For any $r>0$, define \begin{equation*}
  \psi(r)= \begin{cases}
  c_1 r+ \int_0^r e^{-c_2 g(s)}\, ds ,& r\in [0, 2l_0],\\
 \psi(2l_0)+ \psi'(2l_0)(r-2l_0), & r\in(2l_0,\infty).
  \end{cases}
  \end{equation*}
  Note that the function $\psi$ is the same as in the proof of \cite[Theorem 4.2]{Lwcw}. Similarly as there, we can show by a simple calculation that
\begin{equation}\label{e:eth1}\begin{split}
&\psi'(|x-y|) \Big[K_1|x-y|\I_{\{|x-y|\le l_0\}}-K_2|x-y|\I_{\{|x-y|\ge l_0\}} \Big]+\widetilde L_Z\psi(|x-y|) \\
& \le -\lambda_0 \psi(|x-y|),
\end{split}\end{equation}
 where $\widetilde L_Z$ is the refined basic coupling operator for the L\'evy process $Z$, and hence \begin{equation}\label{e:eth2}\begin{split}
    \widetilde L_Z\psi(|x-y|)    &=\frac{1}{2} \mu_{(x-y)_{\kappa}}(\R^d)\Big[\psi\big(|x-y|+\kappa\wedge |x-y|\big)\\
     &\qquad\qquad\qquad\quad+  \psi\big(|x-y|-\kappa\wedge |x-y|\big)- 2\psi(|x-y|)\Big]
  \end{split}
  \end{equation} with
  $\lambda_0=c_1c_2/(1+c_1).$

(ii)
As mentioned in Remark \ref{r:thtpw} (iii) and (iv), from this point onwards the proof deviates substantially from that of \cite[Theorem 4.2]{Lwcw}.
 For any $\delta>0$, consider the equation as follows
\begin{equation}\label{e:coupp}
\begin{split} \begin{cases}
  dX_t=b(X_t,\mu_{X_t})\,dt +dZ_t, & X_0 \sim \mu_{X_0},\\
  dY_t^{\delta}=b(Y_t^{\delta},\mu_{Y_t^\delta})\,dt +dZ_t+dL_t^{*,\delta}, & Y^{\delta}_0 \sim \mu_{Y_0},
  \end{cases}\end{split}\end{equation} where $(L_t^{*,\delta})_{t\ge0}$ is given by \eqref{Ldelta}. To prove the existence of a weak solution $(X_t,Y_t^{\delta})_{t\ge0}$ to the system \eqref{e:coupp}, we consider the following nonlinear operator acting on $f\in C_b^2(\R^{2d})$ (for any fixed $\mu_1$ and $\mu_2\in \mathscr{P}_1(\R^d)$),
  $$\widetilde{L}[\mu_1,\mu_2]f(x,y)=\langle b(x,\mu_1), \nabla_x f(x,y)\rangle+ \langle b(y,\mu_2), \nabla_y f(x,y)\rangle+\widetilde{L}^\delta_Z f(|x-y|),$$
    where $\widetilde{L}^\delta_Z$ is defined by \eqref{truncatedCoupling}. It is obvious that $\widetilde{L}[\mu_1,\mu_2]$ is a coupling operator for the operators $L[\mu_1]$ and $L[\mu_2]$ given by \eqref{e:inf1}.
  Due to the continuity of $b(x, \mu)$, the drift coefficient of the coupling operator $\widetilde{L}[\mu_1,\mu_2]$ is also continuous (with respect to the product metric).
   On the other hand, as shown in \cite[Proposition A.5]{Lwcw}, \eqref{th10--} implies that there is a non-negative measurable function $\rho$ on $\R^d$ such that $\nu(dz)\ge \rho(z)\,dz$ and
    $$\inf_{x\in \R^d:\, |x|\le \kappa_0}\int_{\R^d}\rho(z)\wedge\rho(x+z)\,dz>0.$$
   Moveover, by \cite[(A.3) in the proof of Proposition A.5]{Lwcw}, the function
     \begin{equation}\label{e:continuousOverlapFunction}
   x\mapsto \int_{\R^d}\rho(z)\wedge\rho(x+z)\,dz
   \end{equation}
  is continuous on $\{x\in \R^d:0<|x|\le \kappa_0\}$. Hence without loss of generality, under condition \eqref{th10--} we can consider the refined basic coupling applied only to the component $\rho(z)\,dz$ of the L\'evy measure. Continuity of \eqref{e:continuousOverlapFunction} together with the fact that $\phi_\delta \in C_b^1([0,\infty))$ yields that the coefficients of the operator $\widetilde{L}^\delta_Z$ are continuous. Besides, by \eqref{ee:ffeerr} and the definition of $\widetilde{L}^\delta_Z$, we can see that its coefficients are bounded. This follows from the fact that we only consider the refined basic coupling in the coupling operator $\widetilde{L}^\delta_Z$ when the distance of the two marginal processes is larger than $\delta/2$.
  Hence, according to \cite[Proposition 1.10]{JMW}, there exists a solution,
  belonging to the set of probability measures on $\bar \R^{2d}$ (the standard one point compactification of $\R^{2d}$),
  to the nonlinear martingale problem for the operator $\widetilde{L}[\mu_1,\mu_2]$. (Note that, since the drift coefficient of the coupling operator $\widetilde{L}[\mu_1,\mu_2]$ is not necessarily bounded, the argument of \cite[Proposition 1.10]{JMW} only guarantees the existence of a solution to the martingale problem for $\widetilde{L}[\mu_1,\mu_2]$ on $\bar \R^{2d}$.)   This further along with  \cite[Proposition 1.7]{JMW} yields the existence of a weak solution $(X_t, Y_t^\delta)_{t\ge0}$, taking values in $\bar \R^{2d}$, to the system \eqref{e:coupp}. The explosion time of the process $(X_t, Y_t^\delta)_{t\ge0}$ is
   $$e=\lim_{n\to\infty}\tau_n,\quad \tau_n=\inf\{t\ge0: |X_t|+|Y_t^\delta|\ge n\}.$$ As mentioned above, $\widetilde{L}[\mu_1,\mu_2]$ is a coupling operator for the operators $L[\mu_1]$ and $L[\mu_2]$ given by \eqref{e:inf1}, and so the marginal processes $(X_t)_{t\ge0}$ and $(Y_t^\delta)_{t\ge0}$ enjoy the same law as the solution to the SDE given by  \eqref{non-1}. Since under assumptions in the beginning of this section any (weak) solution to \eqref{non-1} is non-explosive (e.g., see \cite[Theorem 1.2]{Wang}), we have $e=\infty$.  In particular, there exists a weak solution to \eqref{e:coupp}, which indeed takes values on $\R^{2d}$ and is a coupling of the process determined  by  \eqref{non-1}.

  (iii) Recalling that $U_t^{\delta}=X_t-Y_t^{\delta}$, we have
  $$dU_t^{\delta}= (b(X_t,\mu_{X_t})-b(Y^{\delta}_t, \mu_{Y_t^{\delta}}))\,dt- dL_t^{*,\delta},$$
   and hence, arguing as in \eqref{e:coup-func}, we see that
  \begin{equation}\label{e:proof31aux1}
  d\psi(|U_t^{\delta}|)=\bigg[\frac{\psi'(|U_t^{\delta}|)}{|U_t^{\delta}|}\langle U_t^{\delta}, b(X_t,\mu_{X_t})-b(Y_t^{\delta}, \mu_{Y_t^{\delta}}) \rangle + \widetilde L^{\delta}_Z\psi(|U_t^{\delta}|)\bigg]\,dt+dM^{\psi, \delta}_t \,,
  \end{equation}
  where
  \begin{equation*}
  dM_t^{\psi,\delta} := \int_{\R^d\times [0,1]}(\psi(|U_{t-}^{\delta}-S^\delta(U_{t-}^\delta, z,u)|)-\psi(|U_{t-}^{\delta}|))\,\widetilde{N}(dt,dz,du)
  \end{equation*}
 is a martingale. To see this, note that $$\psi(|U_{t-}^{\delta}-S^\delta(U_{t-}^\delta,z,u)|)-\psi(|U_{t-}^{\delta}|) \leq \|\psi'\|_\infty |S^\delta(U_{t-}^\delta,z,u)|,$$ and observe that $|(U_{t}^{\delta})_{\kappa}| \leq \kappa$ for all $t>0$ and that for any $\delta > 0$ the measure $\phi_{\delta}(|U_{t-}^{\delta}|) \rho((U_{t}^{\delta})_{\kappa},z) \,\nu(dz)$ is finite.

 Moreover, according to \eqref{truncatedCoupling}, we have
  \begin{equation*}
 \widetilde{L}^{\delta}_Z \psi(|x-y|) = \widetilde{L}_Z \psi(|x-y|) \cdot \phi_\delta(|x-y|).
 \end{equation*}
 Hence, using the fact that $\psi'$ is decreasing and the definition of $\phi_\delta(r)$, we see that \eqref{e:eth1} implies
 \begin{equation}\label{e:estimateused}\begin{split}
 &\psi'(|x-y|) \Big[K_1|x-y|\I_{\{|x-y|\le l_0\}}-K_2|x-y|\I_{\{|x-y|\ge l_0\}} \Big]+\widetilde L_Z^{\delta}\psi(|x-y|) \\
 & \le -\lambda_0 \psi(|x-y|) + \lambda_0 \psi(|x-y|) \cdot(1-\phi_\delta(|x-y|)) \\
 &\quad + \psi'(|x-y|) \Big[K_1|x-y|\I_{\{|x-y|\le l_0\}}-K_2|x-y|\I_{\{|x-y|\ge l_0\}} \Big] \\
 &\qquad\qquad\qquad \quad\times (1-\phi_\delta(|x-y|)) \\
 &\le -\lambda_0 \psi(|x-y|) + \lambda_0 \psi(\delta) + K_1\psi'(0) (l_0 \wedge \delta).
 \end{split}\end{equation}
 Combining the inequality above with ${\mathbf B(K_1, K_2, l_0; K_3)}$ and \eqref{e:proof31aux1}, we obtain
 \begin{equation}\label{e:ppoo}\begin{split} d\psi(|U_t^{\delta}|) \le&\bigg[{\psi'(|U_t^{\delta}|)}\Big(K_1|U_t^{\delta}|\I_{\{|U_t^{\delta}|\le l_0\}} -K_2|U_t^{\delta}| \I_{\{|U_t^{\delta}|\ge l_0\}}\Big)+\widetilde L_Z^{\delta}\psi(|U_t^{\delta}|)\bigg]\,dt\\
 & +K_3\psi'(|U_t^{\delta}|) W_1(\mu_{X_t},\mu_{Y_t^{\delta}})\,dt +dM_t^{\psi,\delta}\\
 \le&-\lambda_0 \psi(|U_t^{\delta}|)\,dt + \lambda_0 \psi(\delta)\,dt + K_1\psi'(0) (l_0 \wedge \delta)\,dt \\
 &+ K_3\psi'(0) \Ee |U_t^{\delta}|\,dt+ dM_t^{\psi,\delta}\\
 \le&-\left(\lambda_0-({K_3(1+c_1)}/({2c_1}))\right)\psi(|U_t^{\delta}|)\,dt + \lambda_0 \psi(\delta)\,dt \\
 &+ K_1\psi'(0) (l_0 \wedge \delta)\,dt + K_3(1+c_1) \big((\Ee |U_t^{\delta}|)- |U_t^{\delta}|\big) \,dt+ dM_t^{\psi,\delta}\\
 =&-\lambda\psi(|U_t^{\delta}|)\,dt + \lambda_0 \psi(\delta)\,dt + K_1\psi'(0) (l_0 \wedge \delta)\,dt \\
 &+ K_3(1+c_1) \big((\Ee |U_t^{\delta}|)- |U_t^{\delta}|\big) \,dt+ dM_t^{\psi,\delta},\end{split}\end{equation}
 where in the second inequality we again used the fact that $\psi'$ is decreasing on $[0,\infty)$,
 and the last inequality follows from the facts that $\psi'(0)=1+c_1$ and
 $$\sup_{r>0}\frac{ r}{\psi(r)}\le
 \sup_{r>0}\frac{1}{\psi'(r)}=\frac{1}{\psi'(2l_0)}=\frac{1}{2c_1}.$$Note
 that in the argument above we also used the assumption that $\Ee |X_t|<\infty$ for all $t>0$ to ensure the finiteness of
 $\Ee |U_t^{\delta}|$.

(iv) As a consequence of \eqref{e:ppoo}, we find that
\begin{equation*}
\begin{split}
&\Ee \big[e^{\lambda t} \psi(|U_{t}^{\delta}|)\big]\\
&= \Ee \psi(|U_0^\delta|) + \Ee \bigg(\int_0^{t} \big[\lambda e^{\lambda s}\psi(|U_{s}^{\delta}|) +e^{\lambda s}\, d\psi(|U_{s}^{\delta}|)/ds \big] \,ds\bigg) \\
&\le \Ee \psi(|U_0^\delta|) + \Ee \bigg(\int_0^{t} e^{\lambda s}\Big[\lambda\psi(|U_{s}^{\delta}|)\\
&\qquad-\lambda\psi(|U_s^{\delta}|)+  K_3(1+c_1) \big((\Ee |U_s^{\delta}|)- |U_s^{\delta}|\big) + \lambda_0 \psi(\delta) + K_1\psi'(0) (l_0 \wedge \delta) \Big]
\,ds\bigg)\\
&=\Ee \psi(|U_0^\delta|) + \int_0^t e^{\lambda s} \big( \lambda_0 \psi(\delta) + K_1\psi'(0) (l_0 \wedge \delta) \big) ds \\
&= \Ee \psi(|U_0^\delta|) +  \bigg( \lambda_0 \psi(\delta) + K_1\psi'(0) (l_0 \wedge \delta) \bigg) \frac{1}{\lambda} (e^{\lambda t} - 1) \,.
\end{split}
\end{equation*}
Hence we obtain
\begin{equation*}
\Ee \psi(|U_{t}^{\delta}|) \leq e^{-\lambda t} \Ee \psi(|U_0^\delta|) +  e^{-\lambda t} \big( \lambda_0\psi(\delta) + K_1\psi'(0) (l_0 \wedge \delta) \big) \frac{1}{\lambda} (e^{\lambda t} - 1)
\end{equation*}
for any $\delta > 0$. Recall, however, from the discussion in step (ii), that for any $\delta > 0$ the process $(X_t, Y_t^{\delta})_{t \geq 0}$ is a coupling of two copies of $(X_t)_{t \geq 0}$ with initial distributions $\mu_{X_0}$ and $\mu_{Y_0}$, respectively. Hence
\begin{equation*}
W_{\psi}(\mu_{X_t}, \mu_{Y_t}) \leq e^{-\lambda t} \Ee \psi(|X_0 - Y_0|) +  e^{-\lambda t} \big( \lambda_0\psi(\delta) + K_1\psi'(0) (l_0 \wedge \delta) \big) \frac{1}{\lambda} (e^{\lambda t} - 1) \,.
\end{equation*}
 Since the estimate above holds for any $\delta > 0$, we can now take the limit $\delta \to 0$ and, using the continuity of $\psi$, we obtain
\begin{equation*}
W_{\psi}(\mu_{X_t}, \mu_{Y_t}) \leq e^{-\lambda t} \Ee \psi(|X_0 - Y_0|) \,,
\end{equation*}
which, along with the fact that $2c_1r\le \psi(r)\le (1+c_1)r$ for all $r>0$, yields the desired assertion.
\end{proof}

To conclude this
subsection, we explain how to apply Theorem \ref{thtpw} to prove Theorem \ref{Main-1} and then we discuss how to apply Theorem \ref{thtpw} to McKean-Vlasov SDEs with the drift of the form \eqref{e:dri}.

\begin{proof}[Proof of Theorem $\ref{Main-1}$]
    It is easily seen from condition (2-ii) that assumption \eqref{e:sigma} of Theorem \ref{thtpw} holds with $\sigma(r)=c_0r^{1-\alpha}$. In particular,
    $\sigma(r)\in C([0,2l_0])\cap C^2((0,2l_0])$ is nondecreasing and concave, and $\int_0^r \sigma(s)^{-1}\,ds$ is well defined for all $r\in [0,2l_0]$.
     Hence, using Theorem \ref{thtpw} and Remark \ref{r:thtpw}(ii), we obtain the assertion of Theorem \ref{Main-1}.
\end{proof}

\begin{remark}\label{remarkDrift}
    For McKean-Vlasov SDEs with the drift of the form \eqref{e:dri}, we can easily check that condition (1-i), combined with the assumption that
    \begin{equation}\label{coe-2}
    |b_2(x,z) - b_2(y,z')| \leq K_{b_2}(|x-y|+|z-z'|)
    \end{equation}
    holds for some constant $K_{b_2} > 0$ and for all $x$, $y$, $z$ and $z' \in \R^d$, implies (2-i).
    First, we claim that \eqref{coe-2} implies
    that for any $x,y\in \R^d$ and $\mu_1,\mu_2\in \mathscr{P}_1(\R^d)$,
    \begin{equation}\label{e:nottee}\bigg|\int b_2(x,z)\,\mu_1(dz)-\int b_2(y,z)\,\mu_2(dz)\bigg|\le K_{b_2}(|x-y|+W_1(\mu_1,\mu_2)).\end{equation} Indeed,
    let $\pi$ be the coupling of $\mu_1$ and $\mu_2$ for which the infimum in the definition of $W_1(\mu_1, \mu_2)$ is attained. Then, for any $x,y\in \R^d$ and $\mu_1,\mu_2\in
    \mathscr{P}_1(\R^d)$,
    \begin{align*}&\bigg|\int b_2(x,z)\,\mu_1(dz)-\int b_2(y,z)\,\mu_2(dz)\bigg|\\
    &=\bigg|\iint b_2(x,z)\,\pi(dz,dz')-\iint b_2(y,z')\,\pi(dz,dz')\bigg|\\
    &\le \int \int |b_2(x,z)-b_2(y,z')|\,\pi(dz,dz')\\
    &\le K_{b_2}(|x-y|+W_1(\mu_1,\mu_2)). \end{align*} Now,
    according to
    \eqref{e:nottee} and the continuity of $b_1(x)$ on
    $\R^d$, we can see that $b(x,\mu)=b_1(x)+\int b_2(x,z)\,\mu(dz)$ is continuous on $\R^d\times
    \mathscr{P}_1(\R^d).$ On the other hand, it immediately  follows
    from condition (1-i) and \eqref{e:nottee} that
    \eqref{assumptionOneSidedLipschitzForMcKean} is satisfied.
    Moreover, for any $\mu\in \mathscr{P}_1(\R^d)$, due to \eqref{coe-2} again,
   \begin{align*}|b(0,\mu)|\le &|b_1(0)|+\int |b_2(0,z)|\,\mu(dz)\\
   \le & |b_1(0)|+\int (|b_2(0,0)|+K_{b_2}|z|)\,\mu(dz)\le C_0\left(1+\int |z|\,\mu(dz)\right).\end{align*}
     Hence, (2-i)
    holds true.
    We shall note that \eqref{coe-2} holds true for any $b_2(x,z)$ which satisfies condition (1-ii).
    \end{remark}

\subsection{Convergence in $W_1$: the Lyapunov function approach}
 To study exponential convergence of the McKean-Vlasov SDE \eqref{non-1} without assuming the contractivity at infinity condition ${\mathbf B(K_1, K_2, l_0; K_3)}$ on the drift, we will make use of the Lyapunov function approach instead.
Though the arguments below are partly motivated by these in Subsection  \ref{Sec2.3},  we will see that the case for the distribution-dependent drift is much more complex and delicate than that for the distribution-independent drift.
To this end, we will assume that the drift term $b(x,\mu)$ is of the form
\begin{equation}\label{e:drift-m}b(x,\mu)=b_1(x)+b_2(x,\mu),\end{equation} where $b_1(x)$ is such that there exist constants $\lambda$ and $C_0>0$ such that for all $x\in \R^d$,
\begin{equation}\label{e:drfit-a}\langle b_1(x),x\rangle\le - \lambda |x|^2 +C_0.\end{equation}
We assume \eqref{e:drfit-a} in order to be able to construct a Lyapunov function for \eqref{non-1}. Note that the general Lyapunov condition \eqref{e:ly} that we used in the distribution-independent case does not have a straightforward counterpart in the McKean-Vlasov case, cf.\ the proof of Lemma \ref{lem:var2} below, and hence we work directly with \eqref{e:drfit-a}.
Moreover, \eqref{e:drfit-a} seems to be a natural condition in the study of the exponential ergodicity of McKean-Vlasov SDEs (see \cite[Assumption 2.7]{EGZ}).
We will also need the following assumption:

{\noindent{{\bf Assumption ({C})}\it
    \begin{itemize}
        \item[(i)] There are constants $K_1$, $K_2$ and $K_3>0$ such that for all $x_1,x_2\in \R^d$
        with $x_1\neq x_2$ and $\mu_1,\mu_2\in \mathscr{P}_1(\R^d)$,
$$\frac{\langle b_1(x_1)-b_1(x_2),x_1-x_2\rangle}{|x_1-x_2|} \le   K_1 |x_1-x_2|$$ and  $$\frac{\langle b_2(x_1,\mu_1)-b_2(x_2,\mu_2),x_1-x_2\rangle}{|x_1-x_2|} \le   K_2 |x_1-x_2| +K_3W_{1}(\mu_1,\mu_2).$$
        \item[(ii)]There is a constant $B_0>0$ such that for all $x\in \R^d$ and $\mu\in\mathscr{P}_1(\R^d)$,  $$|b_2(x,\mu)|\le B_0\left(1+\int |z|\,\mu(dz)+|x|\right).$$
    \item[(iii)] For any $\theta>0$, there exists a measure $0<\nu_\theta\le \nu$ such that ${\rm {supp}}\,\nu_\theta \subset B(0,1)$, $\int_{\{|z|\le 1\}} |z|\,\nu_\theta(dz)\le\theta,$ and
        $$
        \lim_{r\to 0}\inf_{s\in (0,r]} J_{\nu_\theta}(s)s^\alpha >0,
        $$ where $\alpha:=\alpha(\theta)\in (0,1)$ and $$J_{\nu_\theta}(s):= \inf_{x\in \R^d: |x|\le s} \big[\nu_\theta\wedge (\delta_x\ast \nu_\theta)\big]( \R^d)>0.$$
    \end{itemize}}

It is easy to see that under Assumption {\bf(C)}(i)--(ii), both \eqref{e:assdrift} and \eqref{e:assdrift00} are satisfied. Thus,
according to Proposition \ref{P:exist--}, if $\int_{\{|z|>1\}}|z|\,\nu(dz)<\infty$, then
the SDE \eqref{non-1}
has a unique non-explosive strong solution $(X_t)_{t\ge0}$ such that
$\Ee |X_t|<\infty$ for all $t>0$.

Let us begin by discussing how to construct a Lyapunov function for the  McKean-Vlasov SDE \eqref{non-1} under Assumption {\bf (C)}.

\begin{lemma}\label{lem:var2} Let the drift term  $b(x,\mu)$ be of the form \eqref{e:drift-m}, and satisfy \eqref{e:drfit-a} and Assumption {\bf (C)}$(i)$ and $(ii)$.
Let $\int_{\{ |z| > 1 \}} |z| \,\nu(dz) < \infty$, and let
$V\in C^2(\R^d)$ be a radial function such that $V(x)\ge1$ for all $x\in \R^d$,
$V(x)=1+|x|$ for all $|x|\ge1$
and $\| \nabla V \|_{\infty} + \| \nabla^2 V \|_{\infty} < \infty$.
Then, there is a constant $C>0$ $($independent of $B_0)$ such
that for all $t>0$,
\begin{align*}
dV(X_t)\le& \left[-(\lambda-2\| \nabla V \|_{\infty}B_0)V(X_t)+\| \nabla V \|_{\infty}B_0((\Ee |X_t|)-|X_t|)+C(1+B_0)\right]\,dt \\
&+dM_t^V,
\end{align*}
 where $(M_t^V)_{t\ge0}$ is a
 martingale, $\lambda$ is given in \eqref{e:drfit-a} and $B_0$ is the constant in Assumption {\bf(C)}$(ii)$.
In particular, when $0<2\|\nabla V\|_{\infty}B_0<\lambda$, for any $X_0$ such that $\Ee V(X_0) < \infty$ and for any $t>0$,
\begin{equation}\label{lem:var2--}\Ee V(X_t)\le \Ee V(X_0)e^{-(\lambda-2\| \nabla V \|_{\infty}B_0) t}+C(1+B_0)/(\lambda-2\| \nabla V \|_{\infty}B_0).\end{equation}
\end{lemma}
\begin{proof}
Let $V\in C^2(\R^d)$ be a radial function such that $V(x)\ge1$ for all $x\in \R^d$,
$V(x)=1+|x|$ for all $|x|\ge1$,
 and  $\| \nabla V \|_{\infty} + \| \nabla^2 V \|_{\infty} < \infty$.
 Let $\mu_{X_t}$ be the distribution of $X_t$. Then, by the It\^{o} formula, it holds that
\begin{align*}dV(X_t)=L[\mu_{X_t}]V(X_t)\,dt+dM_t^V,\end{align*}  where $L[\mu]$ is defined in \eqref{e:inf1} and
$$M_t^V := \int_0^t \int_{\R^d} \left( V(X_{s-} + z) - V(X_{s-}) \right) \,\widetilde{N}(ds,dz)$$ is a martingale, thanks to the assumption that $\int_{\{|z|\ge1\}}|z|\,\nu(dz)<\infty$.

Using the mean value theorem and the fact that $\int_{\{|z|\ge1\}}|z|\,\nu(dz)<\infty$ again, we find that
\begin{align*}&\int(V(x+z)-V(x)-\langle\nabla V(x),z\rangle\I_{\{|z|\le
    1\}})\,\nu(dz)\\
&=\int_{\{|z|\le 1\}}(V(x+z)-V(x)-\langle\nabla V(x),z\rangle)\,\nu(dz)+\int_{\{|z|\ge 1\}}(V(x+z)-V(x))\,\nu(dz)\\
&\le \frac{1}{2}\|\nabla^2V\|_\infty \int_{\{|z|\le 1\}}|z|^2\,\nu(dz) + \|\nabla V\|_\infty\int_{\{|z| \ge 1\}}|z|\,\nu(dz)=:C_1,
\end{align*} where in the last step we used the definition of $V$ (i.e., the function $V\in C^2(\R^d)$ such that $\|\nabla V\|_\infty+\|\nabla^2V\|_\infty<\infty$).

On the other hand, by Assumption {\bf (C)}(ii),
$$
\langle b_2(x, \mu) , \nabla V(x) \rangle \leq |b_2(x,\mu)| \| \nabla V \|_{\infty} \leq \| \nabla V \|_{\infty} B_0 \left( 1 + \int |z| \,\mu(dz) + |x| \right).
$$
In order to deal with the term $\langle b_1(x) , \nabla V(x) \rangle$, we write
$$
 \langle b_1(x) , \nabla V(x) \rangle =  \langle b_1(x) , \nabla V(x) \rangle \I_{\{ |x| \leq 1 \}}+  \langle b_1(x) , \nabla V(x) \rangle\I_{\{ |x| > 1 \}}
$$
and recall that $V(x)$ is radial, which means that there exists a function $\bar{V} : [0,\infty) \to [1, \infty)$ such that $V(x) = \bar{V}(|x|)$ for all $x \in \R^d$. Since by Assumption {\bf (C)}(i),
$$
\langle b_1(x) , x \rangle \leq K_1 |x|^2 + |b_1(0)| \cdot |x|,
$$
we have
$$
\langle b_1(x) , \nabla V(x) \rangle \I_{\{ |x| \leq 1 \}} = \frac{\langle b_1(x) , x \rangle}{|x|} \bar{V}'(|x|) \I_{\{ |x| \leq 1 \}} \leq (K_1 + |b_1(0)|) \| \nabla V \|_{\infty}
$$
where we used $\| \bar{V}' \|_{\infty} = \| \nabla V \|_{\infty}$. Moreover, according to \eqref{e:drfit-a},
$$
\langle b_1(x) , \nabla V(x) \rangle \I_{\{ |x| > 1 \}} \leq -\lambda |x|\I_{\{ |x| > 1 \}} + C_0 ,
$$
where we used the facts that $V(x) = |x| + 1$ for $|x| > 1$, which implies $\bar{V}'(|x|) = 1$ for $|x|>1$, and $C_0/|x| \leq C_0$ for $|x| > 1$.

Hence we get
\begin{align*}
L[\mu]V(x) &\leq C_1 + \| \nabla V \|_{\infty} B_0 \left( 1 + \int |z|\, \mu(dz) + |x| \right) + (K_1 + |b_1(0)|) \| \nabla V \|_{\infty} \\
&\quad- \lambda (|x| + 1)\I_{\{ |x| > 1 \}} + \lambda + C_0 \\
&= C_1 + \| \nabla V \|_{\infty} B_0 \left( 1 + \int |z| \,\mu(dz) + |x| \right) + (K_1 + |b_1(0)|) \| \nabla V \|_{\infty} \\
&\quad - \lambda V(x) + \lambda V(x) \I_{\{ |x| \leq 1 \}} + \lambda + C_0 \\
&\leq C_1 + \| \nabla V \|_{\infty} B_0 \left( \int |z|\, \mu(dz) + V(x) \right) + (K_1 + |b_1(0)|) \| \nabla V \|_{\infty} \\
&\quad - \lambda V(x) + \lambda \sup_{x \in B(0,1)} V(x) + \lambda + C_0 + \| \nabla V \|_{\infty} B_0 ,
\end{align*}
where in the second step we used the fact that $V(x) = |x| + 1$ for $|x| > 1$ and the last inequality follows from  $V(x) \geq |x|$ for all $x \in \R^d$.
We conclude that there is a constant $C>0$ given by
$$
C := \max \left( \| \nabla V \|_{\infty} , C_1 + C_0 + (K_1 + |b_1(0)|) \| \nabla V \|_{\infty} + \lambda \left( 1 + \sup_{x \in B(0,1)} V(x) \right) \right)
$$
such that
for all $x\in \R^d$ and $\mu\in \mathscr{P}_1(\R^d)$,
$$
L[\mu] V(x) \le -(\lambda-\| \nabla V \|_{\infty}B_0)V(x)+\| \nabla V \|_{\infty}B_0\int |z|\,\mu(dz)+C(1+B_0).
$$

 With this inequality at hand, we obtain
\begin{align*}dV(X_t)&=L[\mu_{X_t}]V(X_t)\,dt+dM_t^V\\
&\le  [-(\lambda-\| \nabla V \|_{\infty}B_0)V(X_t)+\| \nabla V \|_{\infty}B_0\Ee |X_t|+C(1+B_0)]\,dt+dM_t^V\\
&\le [-(\lambda-2\| \nabla V \|_{\infty}B_0)V(X_t)+\| \nabla V \|_{\infty}B_0((\Ee
|X_t|)-|X_t|)+C(1+B_0)]\,dt\\
&\quad +dM_t^V,\end{align*} where in the last step we used the fact that $|x| \leq V(x)$ for all $x \in \R^d$. This proves the
first assertion. Furthermore, since $\Ee|X_t|<\infty$ for all $t>0$,
we can obtain the second one.
\end{proof}

\begin{remark}\label{remarkLyapunovForStandardSDE}
    Note that when $b_2(x,\mu) = 0$ for all $x\in \R^d$ and $\mu\in \mathscr{P}(\R^d)$ in \eqref{e:drift-m}, the proof of Lemma \ref{lem:var2} still
    works in the same way. Hence, assuming $\int_{\{ |z| > 1 \}} |z| \,\nu(dz) < \infty$ and
    $\langle b(x) , x \rangle \leq - \lambda |x|^2 + C_0$ for all $x \in \R^d$ with some $\lambda,
C_0>0$, we obtain a Lyapunov function $V$ as stated in Lemma \ref{lem:var2} for the distribution independent SDE \eqref{s1}. In particular, this Lyapunov function satisfies both Assumption {\bf (A)}(iii) and Assumption {\bf (B)}(iii).
\end{remark}

  We will now extend Theorem \ref{th2:le} to
  the  McKean-Vlasov  SDE given by \eqref{non-1}.

\begin{theorem}\label{theoremMcKeanVlasovMultiplicative}Suppose that the drift $b(x,\mu)$ is of the form \eqref{e:drift-m} such that \eqref{e:drfit-a} and Assumption {\bf (C)} are satisfied.
	 Suppose also that $\int_{\{ |z| > 1 \}} |z| \nu(dz) < \infty$.
	 Let $V$ be the Lyapunov function from Lemma $\ref{lem:var2}$, and let $\mu_{X_0}$ and $\mu_{Y_0}$ be probability measures such that the integrals $\mu_{X_0}(V)$ and $\mu_{Y_0}(V)$ are finite
$($which is  
equivalent to saying
that both $\mu_{X_0}$ and $\mu_{Y_0}$ have finite first moment$)$.
 Then there are constants $K_2^*$, $K_3^*$ and $ B_0^*>0$ such that for all $K_2\in (0,K_2^*)$, $K_3\in (0,K_3^*)$, $B_0\in (0,B_0^*)$ and $t>0$,
$$W_{\Phi}(\mu_{X_t}, \mu_{Y_t})\le C_0 e^{-\lambda_0t}  \left( \mu_{X_0}(V) + \mu_{Y_0}(V) \right)^2,$$ where
$\Phi(x,y)=(|x-y|\wedge1)(V(x)+V(y))$,
 and $C_0,\lambda_0>0$ are constants independent of $\mu_{X_0}$, $\mu_{Y_0}$ and
$t$.
\end{theorem}
\begin{proof} (i) Let $V$ be the Lyapunov function from Lemma \ref{lem:var2}. Choose $\psi$ as that in Lemma \ref{L:kk1}, and define $F(x,y)=\psi(|x-y|)(1+\varepsilon (V(x)+V(y)))$ for any $x,y\in \R^d$ and some $\varepsilon>0$. Then, following the proof of Theorem \ref{th2:le} and using \eqref{e:drfit-a} and Assumption {\bf (C)}(i) and (iii), we can find constants $\varepsilon, \lambda_0>0$ such that for all $x,y\in \R^d$,
\begin{equation}\label{e:times1}\widetilde L F(x,y)\le -\lambda_0 F(x,y),\end{equation} where $\widetilde L$ is the
operator given by \eqref{e:couplg} with $b$ replaced by $b_1$. Note
that in the argument above, in order to verify Assumption {\bf
(B)}(iii), we used \eqref{e:drfit-a} and the properties of the
function $V$, cf.\ Remark \ref{remarkLyapunovForStandardSDE}.
Note also that the constructions of functions $\psi$ and $F$ are
independent of $b_2(x,\mu)$, and, in particular, the constants
$\varepsilon, \lambda_0$ are independent of $K_2, K_3$ and $B_0$ in
Assumptions {\bf{C}}(i) and (ii).

Furthermore, for any $x_1,x_2\in \R^d$ and $\mu_1,\mu_2\in \mathscr{P}_1(\R^d)$,
\begin{align*}&\frac{\langle b_2(x_1,\mu_1)-b_2(x_2,\mu_2), x_1-x_2\rangle}{|x_1-x_2|} \psi'(|x_1-x_2|)(1+\varepsilon(V(x_1)+V(x_2)))\\
&\quad +\varepsilon\psi(|x_1-x_2|)\big(\langle b_2(x_1,\mu_1),\nabla V(x_1)\rangle+\langle b_2(x_2,\mu_2),\nabla V(x_2)\rangle\big)\\
&\le K_2\psi'(|x_1-x_2|)|x_1-x_2|(1+\varepsilon(V(x_1)+V(x_2)))\\
&\quad+ K_3\psi'(|x_1-x_2|)W_1(\mu_1,\mu_2)(1+\varepsilon(V(x_1)+V(x_2)))\\
&\quad+\varepsilon B_0 \|\nabla V \|_{\infty} \psi(|x_1-x_1|)\left(2+ \int |z|\,\mu_1(dz)+\int |z|\,\mu_2(dz)+|x_1|+|x_2|\right)\\
&\le (K_2+2 B_0 \|\nabla V \|_{\infty}) \psi(|x_1-x_2|)\\
&\qquad\times\left[1+\varepsilon(V(x_1)+V(x_2)) +  \varepsilon\left( \int |z|\,\mu_1(dz)+\int |z|\,\mu_2(dz)\right)\right]\\
&\quad + K_3\psi'(0)W_1(\mu_1,\mu_2)(1+\varepsilon(V(x_1)+V(x_2))),\end{align*} where we used Assumptions {\bf (C)}(i) and (ii) in the first inequality,
and in the second inequality we used $V(x) \geq |x|$ and the fact that $\psi'(r)r\le \psi(r)$ (since $\psi(0)=0$ and $\psi'$ is decreasing).

(ii)
 In the following, let $L$ be the operator given by \eqref{SDE-generator} with $b$ replaced by $b_1$. In particular, $\widetilde L$ above is a coupling operator of $L$.
Motivated by the proof of Theorem \ref{thtpw}, we want to replace the component $\widetilde{L}_Z$ of $\widetilde{L}$ in \eqref{e:couplg} with the generator $\widetilde{L}_Z^{\delta}$ corresponding to the combination of the refined basic coupling and the synchronous coupling, defined by \eqref{truncatedCoupling}.
Note that, for the coupling  operator  $\widetilde{L}^{\delta}$, the equality \eqref{e:estim1} in the proof of Theorem \ref{th2:le} becomes
\begin{align*}\widetilde{L}^{\delta} & F(x,y)\\
=&\widetilde{L}^{\delta} \psi(|x-y|) \cdot (1+\varepsilon(V(x)+V(y))) + \psi(|x-y|)\cdot \varepsilon L(V(x)+V(y))\\
&+\phi_\delta(|x-y|)\cdot\frac{1}{2}\varepsilon(\psi(|x-y-(x-y)_\kappa|)-\psi(|x-y|))\\
&\quad\times\int [(V(x+z)-V(x))+(V(y+z+(x-y)_\kappa)-V(y))]\,\mu_{\theta,(y-x)_\kappa}(dz)\\
&+\phi_\delta(|x-y|)\cdot\frac{1}{2}\varepsilon(\psi(|x-y+(x-y)_\kappa|)-\psi(|x-y|))\\
&\quad\times\int [(V(x+z)-V(x))+(V(y+z-(x-y)_\kappa)-V(y))]\,\mu_{\theta,(x-y)_\kappa}(dz)\\
=&\widetilde{L} F(x,y)\cdot \phi_\delta(|x-y|)+   \varepsilon \psi(|x-y|)L(V(x)+V(y))\cdot (1-\phi_\delta(|x-y|)).\end{align*}
Hence, using \eqref{e:times1} and repeating the reasoning from the proof of Theorem \ref{thtpw}, we obtain
\begin{align*}
\widetilde{L}^{\delta} F(x, y) \le & - \lambda_0 F(x, y) + \lambda_0 F(x, y) \cdot (1-\phi_\delta(|x-y|)) \\
&+  \varepsilon \psi(|x- y|)L(V(x) + V(y))  \cdot (1-\phi_\delta(|x-y|)) \\
&+ K_1\psi'(|x-y|)  |x-y|  \cdot (1-\phi_\delta(|x-y|)),
\end{align*}
where for the last term we used Assumption {\bf C}(i), i.e.,\ $b_1$ is one-sided Lipschitz with the constant $K_1$. Thus, due to the continuity of $\psi$ at zero and the fact that both $\psi'$ and $LV$ are bounded, we see that for any $x,y\in \R^d$,
$$
\widetilde{L}^{\delta} F(x, y) \leq - \lambda_0 F(x,y)+\lambda_0\psi(\delta)(1+\varepsilon (V(x)+V(y))) + C(\delta) \,,
$$
where $C(\delta)$ is independent of $K_2, K_3$ and $B_0$, and
such that $C(\delta) \to 0$ as $\delta \to 0$.

(iii)  Now, as in part (ii) of the proof of Theorem \ref{thtpw}, we can define a coupling process $(X_t,Y_t^\delta)_{t\ge0}$ of the process $(X_t)_{t\ge0}$, by using the system of SDEs \eqref{e:coupp}. Similarly as in the proof of Theorem \ref{th2:le}, we apply the refined basic coupling only to the component $\nu_\theta$ of the L\'evy measure $\nu$. Under Assumption {\bf(C)}, we can verify the existence of a weak solution to \eqref{e:coupp} in the present setting. Furthermore, by the It\^{o} formula,
\begin{align*}
&d F(X_t,Y^\delta_t)\\
&=\widetilde{L}^\delta F(X_t, Y^\delta_t)\, dt \\
&\quad+ \left( \frac{\psi'(|X_t - Y^\delta_t|)}{|X_t - Y^\delta_t|} \langle b_2(X_t, \mu_{X_t}) - b_2(Y^{\delta}_t, \mu_{Y^\delta_t}) , X_t - Y^\delta_t \rangle \right) \\
&\qquad\quad\times(1+ \varepsilon V(X_t) + \varepsilon V(Y^\delta_t))\,dt\\
&\quad+ \varepsilon \psi(|X_t - Y^\delta_t|)\left( b_2 (X_t , \mu_{X_t}) , \nabla V(X_t) \rangle + \langle b_2 (Y^\delta_t , \mu_{Y^\delta_t}) , \nabla V(Y^\delta_t) \rangle \right)\,dt\\
&\quad+ (1 + \varepsilon V(X_t) + \varepsilon V(Y^\delta_t)) \,dM_t^{\psi,\delta} + 2 \varepsilon \psi(|X_t - Y^\delta_t|)\,dM_t^V,
\end{align*} where both $(M_t^{\psi,\delta})_{t\ge0}$ and $(M_t^V)_{t\ge0}$ are martingales.
Using all the estimates in parts (i) and (ii), we can bound the drift component in the inequality above by
\begin{align*}
&-\lambda_0 F(X_t, Y^\delta_t) +\lambda_0\psi(\delta)(1+\varepsilon(V(X_t)+V(Y_t^\delta)))+C(\delta)\\
&\quad +(K_2+2 B_0 \|\nabla V \|_{\infty})\psi(|X_t - Y_t^\delta|) \left( 1 + \varepsilon (V(X_t) + V(Y^\delta_t) + \Ee |X_t| + \Ee |Y^\delta_t|) \right)\\
&\quad+K_3\psi'(0)(\Ee |X_t-Y_t^\delta|)\cdot (1+\varepsilon (V(X_t) + V(Y^\delta_t))) \\
&\leq -\lambda_0 F(X_t, Y_t^\delta) +\lambda_0\psi(\delta)(1+\varepsilon(V(X_t)+V(Y_t^\delta)))+C(\delta)\\
&\quad+ (K_2+2 B_0 \|\nabla V \|_{\infty}) F(X_t, Y_t^\delta)  \\
&\quad + (K_2+2 B_0 \|\nabla V \|_{\infty}) \psi(|X_t - Y_t^\delta|) \varepsilon \left( \Ee |X_t| + \Ee|Y_t^{\delta}| \right) \\
&\quad+K_3\psi'(0)(\Ee |X_t-Y_t^\delta|)\cdot (1+\varepsilon (V(X_t) + V(Y^\delta_t)))\\
&\le -[\lambda_0-(K_2+2 B_0 \|\nabla V \|_{\infty})]F(X_t,Y_t^\delta)\\
&\quad +(K_2+2 B_0 \|\nabla V \|_{\infty}) \psi'(0)|X_t-Y_t^\delta| \left( 1 + \varepsilon \Ee V(X_t) + \varepsilon \Ee V(Y_t^{\delta}) \right) \\
&\quad+ (\lambda_0\psi(\delta)+K_3\psi'(0)\Ee |X_t-Y_t^\delta|)(1+\varepsilon(V(X_t)+V(Y_t^\delta)))+C(\delta),
 \end{align*}
where in the second inequality we used the facts that $V(x) \geq |x|$ for all $x\in \R^d$ and $\psi(r) \leq \psi'(0)r$ for all $r \geq 0$.

Hence, applying all the estimates above, we find that for any
$\lambda_*,t>0$ and for any $X_0 \sim \mu_{X_0}$ and $Y_0^\delta
\sim \mu_{Y_0}$,
\begin{align*} &e^{\lambda_* t}\Ee F(X_t,Y_t^\delta)\\
&\leq \Ee F(X_0,Y_0^\delta)+ \big[\lambda_*-(\lambda_0-(K_2+2 B_0 \|\nabla V \|_{\infty}))\big] \int_0^t e^{\lambda_*s} \Ee F(X_s,Y_s^\delta)\,ds\\
&\quad+\lambda_0\psi(\delta)\int_0^t e^{\lambda_*s}(1+\varepsilon(\Ee V(X_s)+\Ee V(Y_s^\delta)))\,ds\\
&\quad+\left( K_3 + K_2+2 B_0 \|\nabla V \|_{\infty} \right) \psi'(0) \\
&\qquad \times \int_0^t e^{\lambda_*s}(\Ee |X_s-Y_s^\delta|)(1+\varepsilon(\Ee V(X_s)+\Ee V(Y_s^\delta)))\,ds\\
&\quad+C(\delta)\int_0^te^{\lambda_*s}\,ds\\
&=:
\Ee F(X_0,Y_0^\delta)+\sum_{i=1}^4 I_4.  \end{align*}

Next, we will give bounds for $I_2$ and $I_3$.
We first
choose $B_0\in (0,\lambda/(2\|\nabla V\|_{\infty}))$, where $\lambda$ is the constant given
in \eqref{e:drfit-a}. Then, according to Lemma \ref{lem:var2},
\begin{align*}I_2\le& \lambda_0\psi(\delta) \int_0^t e^{\lambda_*s}\\
&\quad \times\left(1+\varepsilon(\Ee V(X_0)+ \Ee V(Y^\delta_0))e^{-(\lambda-2 \|\nabla V\|_{\infty}B_0)s}+\frac{2\varepsilon C(1+B_0)}{\lambda-2\|\nabla V\|_{\infty}B_0}\right)\,ds\\
\leq& \lambda_0 \psi(\delta)  \left(1+\frac{2\varepsilon C(1+B_0)}{\lambda-2\|\nabla V\|_{\infty}B_0}+ \varepsilon \mu_{X_0}(V)
+ \varepsilon \mu_{Y_0}(V)\right)\int_0^t  e^{\lambda_*s}\,ds \\
=&: \widetilde{C}(\delta) \int_0^t  e^{\lambda_*s}\,ds, \end{align*}
where $\widetilde{C}(\delta)$ satisfies $\widetilde{C}(\delta) \to 0$ as $\delta \to 0$.
We will now proceed with bounding $I_3$.
In the arguments
below, the constants $C_i$ $(i=1,2,\ldots,4)$ are all independent of
$X_0$, $Y_0^\delta$, $t$, $\delta$, $K_2$, $K_3$ and $B_0$.
By the definition of $F(x,y)$ and the properties of
$V$, there is a constant $C_1\ge 1$ such that for all $x,y\in \R^d$,
$$C_1^{-1}|x-y|\le F(x,y)\le C_1 (V(x)+V(y)).$$ According to Lemma \ref{lem:var2} again, for all $s>0$,
\begin{equation}\label{e:auxThm35_2}
\begin{split}
1+&\varepsilon(\Ee V(X_s)+\Ee V(Y_s^\delta)) \\ &\leq C_2\left(1+\frac{1+B_0}{\lambda-2\|\nabla V\|_{\infty}B_0}\right)+C_2\Ee(V(X_0)+V(Y^{\delta}_0))e^{-(\lambda-2\|\nabla V\|_{\infty}B_0)s}.
\end{split}
\end{equation}
Hence,
\begin{align*}I_3&\le C_3\left( K_3 + K_2+  B_0  \right)\left(1+\frac{1+B_0}{\lambda-2\|\nabla V\|_{\infty}B_0}\right)\int_0^t e^{\lambda_*s} \Ee F(X_s,Y_s^\delta)\,ds\\
&\quad+ C_3\left( K_3 + K_2+  B_0   \right)\Ee (V(X_0)+V(Y_0^{\delta})) \int_0^t e^{[\lambda_*-(\lambda-2\|\nabla V\|_{\infty}B_0)]s}\Ee F(X_s,Y_s^\delta)\,ds\\
&\le C_3\left( K_3 + K_2+ B_0
\right)\left(1+\frac{1+B_0}{\lambda-2\|\nabla V\|_{\infty}B_0}\right)\int_0^t
e^{\lambda_*s}
\Ee F(X_s,Y_s^\delta)\,ds\\
&\quad + C_4\left( K_3 + K_2+  B_0
\right)\left(1+\frac{1+B_0}{\lambda-2\|\nabla V\|_{\infty}B_0}\right)\\
&\qquad\times \left( \Ee (V(X_0) + V(Y_0^{\delta}))
\right)^2\int_0^t e^{[\lambda_*-(\lambda-2\|\nabla V\|_{\infty}B_0)]s}\,ds, \end{align*}
where in the first inequality we used the fact that $|x-y| \leq C_1 F(x,y)$ for all $x$, $y \in \R^d$ and \eqref{e:auxThm35_2}, while in the second inequality we used the fact that $F(x,y)\le C_1(V(x)+V(y))$ for all $x,y\in \R^d$ and Lemma \ref{lem:var2}
again, as well as $$\Ee(V(X_0) + V(Y_0^{\delta})) \leq \left(
\Ee(V(X_0) + V(Y_0^{\delta})) \right)^2$$ since $V \geq 1$.

Combining all the estimates above, we arrive at
\begin{equation}\label{e:auxThm35}
\begin{split}
&e^{\lambda_* t}\Ee F(X_t,Y_t^\delta)\\
&\leq \Ee F(X_0,Y_0^\delta)\\
&\quad+\bigg[\lambda_*+K_2+  2B_0\|\nabla V\|_\infty +C_3\left( K_3 +
K_2+ B_0
\right)\left(1+\frac{1+B_0}{\lambda-2\|\nabla V\|_{\infty}B_0}\right)  \\
&\qquad\quad-\lambda_0\bigg] \int_0^t e^{\lambda_*s} \Ee F(X_s,Y_s^\delta)\,ds\\
&\quad+C_5(K_2,K_3,B_0)\left( \mu_{X_0}(V) + \mu_{Y_0}(V) \right)^2\int_0^te^{(\lambda_*+2\|\nabla V\|_{\infty}B_0-\lambda)s}\,ds\\
&\quad
+ \big( C(\delta) + \widetilde{C}(\delta) \big)\int_0^te^{\lambda_*s}\,ds.
\end{split}
\end{equation} Here and in what
follows, the constants $C_i(K_2,K_3,B_0)$ ($i=5,6,7$) may
depend on $K_2,K_3$ and $B_0$. Choosing $K_2^*$, $K_3^*$ and $B_0^*$
small enough so that
$$ K^*_2+  2B^*_0\|\nabla V\|_\infty
 +C_3\left( K^*_3 +
K^*_2+ B^*_0 \right)\left(1+\frac{1+B^*_0}{\lambda-2\|\nabla V\|_{\infty}B^*_0}\right)\le
\lambda_0/2,$$
and then letting $\lambda_* \leq \lambda_0/2$,
we can bound the second term on the right hand side of \eqref{e:auxThm35} by zero. Moreover, if we choose $\lambda_* < \lambda - 2 \| \nabla V \|_{\infty} B_0$, then we can bound the integral in the third term by a constant independent of $t$.
 As a consequence, we find
that for any $K_2\in (0,K_2^*)$, $K_3\in (0,K_3^*)$, $B_0\in
(0,B_0^*)$, $\lambda_* < \min \left\{ \lambda_0/2 , \lambda - 2 \| \nabla V \|_{\infty}
 B_0^*\right\}$ and $t>0$,
\begin{align*}\Ee F(X_t,Y_t^\delta)\le & C_6(K_2,K_3,B_0)  e^{-\lambda_* t}\left( \Ee F(X_0,Y_0^\delta)+ \left( \mu_{X_0}(V) + \mu_{Y_0}(V) \right)^2 \right) \\
&+ \big( C(\delta) + \widetilde{C}(\delta) \big)/\lambda_*\\
\le& C_7(K_2,K_3,B_0)e^{-\lambda_* t}\left( \mu_{X_0}(V) +
\mu_{Y_0}(V) \right)^2+ \big( C(\delta) + \widetilde{C}(\delta) \big)/\lambda_*,
\end{align*} where in the second inequality we
used the fact that $$\Ee F(X_0, Y_0^{\delta}) \leq C_1 \left(
\mu_{X_0}(V) + \mu_{Y_0}(V) \right)\le C_1 \left( \mu_{X_0}(V) +
\mu_{Y_0}(V) \right)^2.$$ Letting $\delta\to0$ in the inequality
above, we can prove the desired assertion by following the arguments
from the final part of the proof of Theorem \ref{thtpw}.
\end{proof}

At the end of this part, we give the
proof of Theorem \ref{Main-1add}.
\begin{proof}[Proof of Theorem $\ref{Main-1add}$]
	Combining Theorem \ref{theoremMcKeanVlasovMultiplicative} with the fact that for all probability measures $\mu_1$, $\mu_2 \in \mathscr{P}_1(\R^d)$ we have $W_1(\mu_1,\mu_2) \leq CW_{\Phi}(\mu_1,\mu_2)$ for some constant $C > 0$ (cf.\ the proof of Corollary \ref{co2:le}), we obtain
	 \begin{equation}\label{eq:proof15}
	 W_1(\mu_{X_t}, \mu_{Y_t})\le C_1 e^{-\lambda_0t}  \left( \mu_{X_0}(V) + \mu_{Y_0}(V) \right)^2
	 \end{equation}
	  for some constant $C_1 > 0$.
	  Here the constant $\lambda_0 > 0$, the function $V$ and the measures $\mu_{X_0}$, $\mu_{Y_0}$, $\mu_{X_t}$ and $\mu_{Y_t}$ are as in the statement of Theorem \ref{theoremMcKeanVlasovMultiplicative}. Similarly as in the proof of
Theorem \ref{Main-1}, we can follow e.g.\ the proof of
 \cite[Theorem 3.1(2)]{Wang} and use \eqref{lem:var2--} to conclude that there exists a
 unique invariant measure $\mu$ for the solution $(X_t)_{t \geq 0}$ to (\ref{non-1}). Hence we can consider our bound (\ref{eq:proof15}) with $(Y_t)_{t \geq 0}$ initiated at $\mu_{Y_0} = \mu$, which concludes the proof.
\end{proof}

\section{Propagation of chaos}\label{section4}

This section is devoted to the
proof of Theorem \ref{Main-4}.

\begin{proof}[Proof of Theorem $\ref{Main-4}$] The proof is based on the
idea of the proof for \cite[Theorem 2]{DEGZ}. For convenience, we denote $\tilde b_2 = b_2$, i.e., we assume that the drift $b(x,\mu)$ is given by
$$b(x, \mu) = b_1(x) + \int b_2(x-z) \,\mu(dz) = b_1(x) + b_2 \ast \mu(x).$$
In particular, we denote $K_{\tilde b_2} = K_{b_2}$ in condition \eqref{coe-b3}.
      We first follow part (i) of the proof of Theorem \ref{thtpw}. Based on the contractivity at infinity condition \eqref{cond:drif1} for $b_1(x)$, we construct a function $\psi$ as in Theorem \ref{thtpw}, with constants $K_1$, $K_2$, $l_0$ replaced by $K_{1,b_1}$, $K_{2,b_1}$, $r_{b_1}$, respectively.
      Now observe that due to \eqref{assumption:concentrationLevymeasure}, for $l_0 = r_{b_1}$ we can find a function $\sigma(r)$ satisfying condition \eqref{e:sigma} for $r \in (0,2r_{b_1}]$, cf. the proof of \cite[Theorem 1.1]{Lwcw}. Hence, proceeding as in \eqref{e:eth1}, we have
      \begin{equation}\label{e:auxProp1}
      \begin{split}
      &\psi'(|x-y|) \Big[K_{1,b_1}|x-y|\I_{\{|x-y|\le r_{b_1}\}}-K_{2,b_1}|x-y|\I_{\{|x-y|\ge r_{b_1}\}} \Big]+\widetilde L_Z\psi(|x-y|) \\
      & \le -\lambda_0 \psi(|x-y|),
      \end{split}
      \end{equation}
      where $\widetilde{L}_Z$ is the refined basic coupling operator for $Z$ and $\lambda_0 > 0$ is an explicit constant defined as in the proof of Theorem \ref{thtpw}, depending only on $K_{1,b_1}$, $K_{2,b_1}$, $r_{b_1}$ and properties of the L\'{e}vy measure $\nu$.
      In particular, the construction of $\psi$ and the constant $\lambda_0$ is independent of $K_{b_2}$.

Motivated by part (ii) of the proof of Theorem \ref{thtpw}, for any fixed $\delta > 0$ we consider the following coupling between the independent
McKean-Vlasov processes and the mean-field particle system
\begin{align*} \begin{cases}
  d \bar X_t^i=b_1(\bar X_t^i)+b_2*\mu_{\bar X_t^i}(\bar X_t^i)\,dt +d Z^i_t,\\
  d X_t^{i,n,\delta}=b_1(X_t^{i,n,\delta})+\frac{1}{n}\sum_{j=1}^nb_2(X_t^{i,n,\delta}-X_t^{j,n,\delta})\,dt +dZ^i_t+dL_t^{*,i,\delta},\quad 1\le i\le n,
  \end{cases}\end{align*} where
  $(L_t^{*,i,\delta})_{t\ge0}$ is defined by analogy to
  \eqref{Ldelta}, with $U_t^\delta$ replaced by $U_t^{i, \delta} := \bar X_t^i-X_t^{i,n,\delta}$.
  Here we assume that $(\bar X_0^i, X_0^{i,n,\delta})$ are independent random variables with the same law having finite second moments, and $(Z_t^i)_{t\ge0}$ are independent L\'evy processes associated with the same L\'evy measure $\nu$.
 We can follow step (ii) in the proof of Theorem \ref{thtpw} to show that the system above has a
 weak solution $(\bar X^i_t,X_t^{i,n,\delta})_{t\ge0}$ for all $1\le i\le n$.

  Below, let $U_t^{i,\delta}=\bar X_t^i-X_t^{i,n,\delta}$ for all $t>0$. Then, by \eqref{e:auxProp1}
    and \eqref{e:estimateused},
    for all $t>0$ and $1\le i\le n$,
  \begin{align*}d\psi(|U_t^{i,\delta}|)
  &\le \psi'(|U_t^{i,\delta}|)\frac{\langle b_1(\bar X_t^i)-b_1(X_t^{i,n,\delta}), U_t^{i,\delta}\rangle}{|U_t^{i,\delta}|}\,dt+\widetilde L_Z^\delta \psi(|U_t^{i,\delta}|)\,dt\\
  &\quad +\psi'(|U_t^{i,\delta}|)\frac{\langle b_2*\mu_{\bar X_t^i}(\bar X_t^i)-\frac{1}{n}\sum_{j=1}^nb_2(X_t^{i,n,\delta}-X_t^{j,n,\delta}), U_t^{i,\delta}\rangle}{|U_t^{i,\delta}|}\,dt+dM_t^{\psi,i,\delta}\\
&\le -\lambda_0 \psi(|U_t^{i,\delta}|)+\psi'(|U_t^{i,\delta}|) A^{i,\delta}_t\,dt+dM_t^{\psi,i,\delta} \\
&\quad + \lambda_0 \psi(|U_t^{i,\delta}|) \left(1 - \phi_{\delta}(|U_t^{i,\delta}|)\right) + \psi'(0) K_{1,b_1} |U_t^{i,\delta}| \left(1 - \phi_{\delta}(|U_t^{i,\delta}|)\right),
\end{align*} where $\psi$ and $\lambda_0$ are as described above,
 $\widetilde L_Z^\delta$ is
defined by \eqref{truncatedCoupling},
$(M_t^{\psi,i,\delta})_{t\ge0}$ is a martingale, and
  $$A^{i,\delta}_t:= \left|b_2*\mu_{\bar X_t^i}(\bar X_t^i)-\frac{1}{n}\sum_{j=1}^nb_2(X_t^{i,n,\delta}-X_t^{j,n,\delta})\right|.$$
In the argument above we used the facts that $\psi'$ is
decreasing and that $b_1(x)$ is one-sided Lipschitz with the
constant $K_{1,b_1}$.
We note again that the construction of $\psi$ and the constant
$\lambda_0$ is independent of $K_{b_2}$.

For notational convenience, let us drop the superscript $\delta$ for now (i.e., denote $A_t^i := A_t^{i,\delta}$, $X_t^{i,n} := X_t^{i,n,\delta}$ and $U_t^i := U_t^{i, \delta}$) and estimate $A_t^i$ as follows
  \begin{align*}A^i_t&\le \frac{1}{n}\sum_{j=1}^n|b_2(\bar X_t^i-\bar X_t^j)-b_2(X_t^{i,n}-X_t^{j,n})|+ \left|b_2*\mu_{\bar X_t^i}(\bar X_t^i)-\frac{1}{n}\sum_{j=1}^nb_2(\bar X_t^{i}-\bar X_t^{j})\right|\\
  &\le \frac{\eta}{n}\sum_{j=1}^n\left(\psi (|U_t^i|)+\psi(|U_t^j|)\right)+ \left|b_2*\mu_{\bar X_t^i}(\bar X_t^i)-\frac{1}{n}\sum_{j=1}^nb_2(\bar X_t^{i}-\bar X_t^{j})\right|\\
  &=:\frac{\eta}{n}\sum_{j=1}^n\left(\psi (|U_t^i|)+\psi(|U_t^j|)\right)+ A_t^{i,*},  \end{align*} where $\eta:=\eta(K_{b_2})=K_{b_2}/c_1$ with the constant $c_1$ defined as in Theorem \ref{thtpw} (independent of $K_{b_2}$), and in the second inequality we used condition (1-ii) and the fact that $c_1r \le \psi(r)$ for all $r>0$.
  Hence, according to both estimates above and the fact that $\psi'(r)\le 1+c_1$ for all $r>0$, we arrive at
  \begin{equation*}\label{e:proofch1} \begin{split}\frac{1}{n}\sum_{i=1}^n\frac{d \Ee\psi(|U_t^i|)}{dt} \le &-\frac{\lambda_0-2(1+c_1)\eta}{n}\sum_{i=1}^n \Ee\psi(|U_t^i|)\,dt+\frac{(1+c_1)}{n}\sum_{i=1}^n \Ee A_t^{i,*}+C(\delta),\end{split}\end{equation*} where $C(\delta)\to0$ as $\delta\to 0$.
  In particular,
  choosing
  \begin{equation}\label{e:auxProp2}
  0<K_{b_2} < \lambda_0c_1/(2(1+c_1)),
  \end{equation}
        we find that \begin{equation}\label{e:proofch1} \begin{split}\frac{1}{n}\sum_{i=1}^n\frac{d \Ee\psi(|U_t^i|)}{dt} \le &-\frac{\lambda}{n}\sum_{i=1}^n \Ee\psi(|U_t^i|)\,dt+\frac{(1+c_1)}{n}\sum_{i=1}^n \Ee A_t^{i,*}+C(\delta),\end{split}\end{equation} where $\lambda:=\lambda_0-2(1+c_1)\eta>0$ due to \eqref{e:auxProp2}.

 Furthermore, given $\bar X_t^i$, the random variables $\bar X_t^j,$ $j\neq i$, are i.i.d. with the same law $\mu_{\bar X_t^i} =: \mu_{X_t}$. In particular, $$\Ee(b_2(\bar X_t^i-\bar X_t^j)|\bar X_t^i)=b_2*\mu_{\bar X^i_t}(\bar X^i_t).$$ Due to condition (1-ii) (particularly, $b_2(0)=0$ and $|b_2(z)|\le K_{b_2}|z|$ for all $z\in \R^d$),
  \begin{align*}\Ee A_t^{i,*}&\le \Ee\left| b_2*\mu_{\bar X_t^i}(\bar X_t^i)-\frac{1}{n-1}\sum_{j=1}^nb_2(\bar X_t^{i}-\bar X_t^{j})\right|\\
  &\quad +\left(\frac{1}{n-1}-\frac{1}{n}\right)\sum_{j=1}^n\Ee| b_2(\bar X_t^{i}-\bar X_t^{j})|\\
  &\le \left(\Ee \left(\Ee\left( \left|b_2*\mu_{\bar X^i_t}(\bar X^i_t)-\frac{1}{n-1}\sum_{j=1}^nb_2(\bar X_t^{i}-\bar X_t^{j})\right|^2\bigg| \bar X_t^i\right)\right)\right)^{1/2}\\
  &\quad+ \sqrt{2}\left(\frac{1}{n-1}-\frac{1}{n}\right)nK_{b_2} \left(\int_{\R^d}|z|^2\,\mu_{X_t}(dz)\right)^{1/2} \\
  &\le \left(\Ee\left(\frac{1}{n-1}\var_{\mu_{X_t}}(b_2(\bar X_t^i-\cdot))\right)\right)^{1/2}\\
  &\quad+ \sqrt{2}\left(\frac{1}{n-1}-\frac{1}{n}\right)nK_{b_2} \left(\int_{\R^d}|z|^2\,\mu_{X_t}(dz)\right)^{1/2} \\
  &\le K_{b_2}\left(\frac{1}{\sqrt{n-1}} + \frac{\sqrt{2}}{n-1}\right)\left(\int_{\R^d}|z|^2\,\mu_{X_t}(dz)\right)^{1/2}. \end{align*}
We now need to bound the second moment of $(X_t)_{t \geq 0}$.
 To this end, consider \begin{equation}\label{e:auxProp3}
\begin{split}
L[\mu] f(x)=&\langle b_1(x), \nabla f(x)\rangle+ \bigg\langle\int b_2(x-z)\,\mu(dz), \nabla f(x)\bigg\rangle\\
&+\int(f(x+z)-f(x)-\langle\nabla f(x), z\rangle\I_{\{|z|\le
1\}})\,\nu(dz),
\end{split}
\end{equation} (cf.\ \eqref{e:inf1}) and apply it to $f(x) = |x|^2$. According to condition (1-i), there is a constant $C_1>0$ such that
$$\langle b_1(x),\nabla f(x) \rangle = 2 \langle b_1(x),x\rangle\le -2K_{2,b_1}|x|^2+2|b_1(0)||x|+C_1,\quad x\in \R^d.$$
Moreover, for the second term in \eqref{e:auxProp3}, due to the fact that $|b_2(z)| \leq K_{b_2}|z|$ for all $z \in \R^d$, we get
\begin{align*}
\bigg\langle\int b_2(x-z)\,\mu(dz), \nabla f(x)\bigg\rangle&=2 \int \langle b_2(x-z) , x \rangle \,\mu(dz) \\
&\leq 2 K_{b_2} \int |x-z|\cdot |x| \,\mu(dz)\\
& \leq 2 K_{b_2} |x|^2 + 2 K_{b_2} |x| \int |z| \,\mu(dz) \,.
\end{align*}
 Finally, since $\int_{\R^d}|z|^2\,\nu(dz)<\infty$, there is a constant $C_2>0$ such that for all $x\in \R^d$,
\begin{align*} &\int(f(x+z)-f(x)-\langle\nabla f(x), z\rangle\I_{\{|z|\le
1\}})\,\nu(dz)\\
&=\int_{\{|z|\le 1\}} \left(|x+z|^2-|x|^2-2\langle x,z\rangle\right)\,\nu(dz)+ \int_{\{|z|>1\}} \left(|x+z|^2-|x|^2\right)\,\nu(dz)\\
&\le C_2(1+|x|).\end{align*}
Hence, we get that there is a constant $C_3 > 0$ such that
\begin{align*}
L[\mu] f(x)&\le (-2K_{2,b_1}+2K_{b_2})|x|^2+2K_{b_2}|x|\int |z|\,\mu(dz)\\
&\quad +(2|b_1(0)|+C_2)|x|+C_1+C_2\\
&\le (-2K_{2,b_1}+4K_{b_2})|x|^2+K_{b_2}\int |z|^2\,\mu(dz)+C_3,
\end{align*} where in the second inequality we used the facts that $2ab\le a^2+b^2$ for all $a,b\ge0$, $\left(\int |z|\,\mu(dz)\right)^2\le \int |z|^2\,\mu(dz)$, and we chose the constant $C_3 > 0$ so that
$$(2|b_1(0)|+C_2)|x| + C_1 + C_2 \le  K_{b_2}|x|^2+C_3,\quad x \in \R^d.$$
Using the inequality above and following
the argument of Lemma \ref{lem:var2}, we can get that for any
$0<K_{b_2}<2K_{2,b_1}/5$,
$$\int_{\R^d}|z|^2\,\mu_{X_t}(dz)=\Ee |X_t|^2\le \Ee|X_0|^2 +C_4$$ holds with some constant $C_4>0$, depending only on the second moment of the L\'{e}vy measure $\nu$, $b_1(0)$ and the constants in condition \eqref{cond:drif1}. Hence, for all $1\le i\le n$,
$$\Ee A_t^{i,*}\le  \frac{C_5}{\sqrt{ n-1}},$$
where $C_5 := K_{b_2} (\sqrt{2} + 1) \left( \Ee |X_0|^2 + C_4 \right)^{1/2}$.
This, along with \eqref{e:proofch1}, yields that if $K_{b_2} <\min\{2K_{2,b_1}/5, \lambda_0c_1/(2(1+c_1))\}$, then
$$\frac{1}{n}\sum_{i=1}^n\frac{d \Ee\psi(|U_t^{i,\delta}|)}{dt} \le -\frac{\lambda}{n}\sum_{i=1}^n \Ee\psi(|U_t^{i,\delta}|)\,dt+\frac{C_5}{\sqrt{n-1}} + C(\delta).$$
Therefore,
$$ \frac{1}{n}\sum_{i=1}^n \Ee\psi(|U_t^{i}|)\le e^{-\lambda t}\frac{1}{n}\sum_{i=1}^n \Ee\psi(|U_0^{i}|)+\frac{C_5}{\lambda\sqrt{n-1}} + \widetilde{C}(\delta),$$
where $\widetilde{C}(\delta) \to 0$ as $\delta \to 0$.
Combining this with the facts that
$W_\psi(\mu_t, \mu_t^n)\le  \Ee\psi(|U_t^{i}|)$ for all $1\le i\le n$ and $c_1r\psi(r)\le (1+c_1)r$ yields the desired assertion.
\end{proof}

\ \

\medskip

\noindent \textbf{Acknowledgements.}
We thank the two anonymous referees for their helpful
comments and valuable corrections that have led to significant improvements of the presentation in the article. The research of Mingjie Liang is supported by the Science Foundation of the Education Department of Fujian Province (No.\ JAT190701) and the National Science Foundation of Sanming University (No.\ B201914). The majority of this work was completed when Mateusz B.\ Majka was affiliated to the University of Warwick and supported by the EPSRC grant no.\ EP/P003818/1.
The research of Jian Wang is supported by the National Natural Science Foundation of China (Nos.\ 11831014 and 12071076),  the Program for Probability and Statistics: Theory and Application (No.\ IRTL1704) and the Program for Innovative Research Team in Science and Technology in Fujian Province University (IRTSTFJ).

\end{document}